\documentclass[11pt]{article}
\usepackage[a4paper,top=3cm,bottom=3.5cm,left=2.5cm,right=2.5cm,marginparwidth=1.75cm]{geometry}
\usepackage[T1]{fontenc}
\usepackage{graphicx}
\usepackage[title]{appendix}
\usepackage{tikz}
\usepackage{pgfplots}
\usepackage{cutwin}

\usepackage{amsmath,yhmath}
\usepackage{amsthm}
\usepackage{amssymb}
\usepackage{hyperref}
\usepackage{cleveref}
\usepackage{epstopdf}
\usepackage{comment}
\usepackage{todonotes}
\usepackage{color}
\usepackage{float}
\usepackage[sort&compress,numbers]{natbib}
\usetikzlibrary{shapes.geometric}
\usetikzlibrary{arrows,automata,quotes}
\usepackage{iftex}
\usepackage{xfrac} 
\usepackage{amsmath,yhmath}
\usepackage{amsthm}
\usepackage{amssymb}
\usepackage{mathtools}
\usepackage{nccmath}
\usepackage{hyperref}
\usepackage{caption, subcaption}
\usepackage[title]{appendix}
\usepackage{cleveref}
\usepackage{epstopdf}
\usepackage{comment}
\usepackage{todonotes}
\usepackage{color}
\usepackage{float}
\usepackage{cutwin}
\usepackage[sort&compress,numbers]{natbib}
\usepackage{makecell}
\usepackage{array}   
\usepackage{multirow}
\usepackage{graphicx}
\usepackage{caption}
\usepackage{subcaption}
\newtheorem{thm}{Theorem}[section]

\newtheorem{lemma}[thm]{Lemma}
\newtheorem{prop}[thm]{Proposition}
\theoremstyle{definition}
\newtheorem{definition}{Definition}[section]

\newtheorem{rem}[definition]{Remark}

\newcommand{\p}{\partial}
\newcommand{\ds}{\displaystyle}

\newcommand{\Z}{\mathbb{Z}}

\newcommand{\srs}{s}
\newcommand{\srsf}{s}
\newcommand{\Srs}{S}
\newcommand{\Srsc}{\mathcal{S}}

\newcommand{\svdots}{\raisebox{0pt}{\scalebox{.6}{$\vdots$}}}
\newcommand{\sddots}{\raisebox{0pt}{\scalebox{.6}{$\ddots$}}}
\newcommand{\sadots}{\raisebox{0pt}{\scalebox{.6}{$\adots$}}}

\numberwithin{equation}{section}
\numberwithin{figure}{section}
\title{Space-Time Wave Localisation in Systems of Subwavelength Resonators}
\author{Habib Ammari\thanks{\footnotesize Department of Mathematics, ETH Z\"urich, R\"amistrasse 101, CH-8092 Z\"urich, Switzerland (habib.ammari@math.ethz.ch, liora.rueff@sam.math.ethz.ch).} \and Erik Orvehed Hiltunen\thanks{\footnotesize Department of Mathematics, University of Oslo, Moltke Moes vei 35, 0851 Oslo, Norway (erikhilt@math.uio.no).} \and Liora Rueff\footnotemark[1] }
\date{}

\begin{document}
	\maketitle
	\begin{abstract}
		In this paper we study the dynamics of metamaterials composed of high-contrast subwavelength resonators and  show the existence of localised modes in such a setting. A crucial assumption in this paper is time-modulated material parameters. We prove a so-called capacitance matrix approximation of the wave equation in the form of an ordinary differential equation. These formulas set the ground for the derivation of a first-principles characterisation of localised modes in terms of the generalised capacitance matrix. Furthermore, we provide numerical results supporting our analytical results showing for the first time the phenomenon of space-time localised waves in a perturbed time-modulated metamaterial. Such spatio-temporal localisation is only possible in the presence of subwavelength resonances in the unperturbed structure. We introduce the time-dependent degree of localisation to quantitatively determine the localised modes and provide a variety of numerical experiments to illustrate our formulations and results.
	\end{abstract}
	\noindent{\textbf{Mathematics Subject Classification (MSC2000):}} 35L51, 35P25, 35C20, 74J05.
	\vspace{0.2cm}
	
	\section{Introduction}
	
	Metamaterials are materials with a repeating micro-structure that exhibit properties surpassing those of the individual building blocks \cite{kadic20193d}. Subwavelength resonators (with high-contrast material parameters and typical size much smaller than the operating wavelength) are used as the building blocks for such large, complex structures which can exhibit a variety of exotic and useful properties, such as localisation and transport properties. A variety of phenomena that can occur when the material parameters vary periodically in the spatial variable such as Dirac degeneracy, topologically protected edge and interface modes,  exceptional point degeneracy, Anderson localisation, Fano-resonance anomaly, bound state in the continuum and single and double near-zero effective properties have been mathematically studied; see \cite{ammari.davies.ea2021,cbms} and the references therein. 
	\par
	If the material parameters of the building blocks of metamaterials depend also periodically on time, there are conceptually similar phenomena since mathematically we can treat the time variable in a similar fashion as the spatial one. 
	Nevertheless, such phenomena have fundamentally different physical implications. 
	Temporal modulation extends the degree of freedom in wave control and manipulation at subwavelength scales and opens further possibilities for extraordinary wave propagation properties such as non-reciprocity and unidirectional amplification \cite{fink}. Structures based on these principles have emerged in the field of \emph{space--time metamaterials}, which represents one of the most active frontiers in the field of metamaterials. \par 
	Although very significant experimental and numerical modelling advances have been recently achieved in the field of time-modulated metamaterials (see, for instance, \cite{Sharabi_Lustig_Segev, GaliffiYinAlú+2022+3575+3581,Ramezani_Jha_Wang_Zhang,josa1,josa2}), very little is known from a mathematical point of view. As far as we know, approximation formulas for the subwavelength resonant frequencies corresponding to time-dependent subwavelength resonators have been derived in \cite{erik_JCP,jinghao_liora} and the properties of the scattered wave field emerging from a time-dependent material have been studied in \cite{ammari2024scattering,florian,erik_bryn_3dscattering}.\par
	
	In this work, we study the possibility of space-time localisation in systems of subwavelength resonators. We  specifically consider one dimensional  metamaterials consisting of subwavelength resonators with periodically time-dependent material parameters and, additionally, time-dependent sources on the resonator boundaries. Our model allows for more specific formulations and applications not only in the field of metamaterials, but also in condensed matter physics due to its analogies with quantum mechanical phenomena. This is due to the fact that, in contrast to higher dimensional cases, interactions between the resonators in one dimensional systems only imply the nearest neighbours. Thus, the capacitance matrix formalism in one dimension is analogous to the tight-binding approximation for quantum systems, while in three dimensions some correspondence holds only for dilute resonators \cite{francesco}.\par 
	
	Space-time localisation at subwavelength scales is a research direction of great scientific interest. It is of particular significance as it enables overcoming the diffraction limit \cite{ammari2023anderson,ammari_convrates,ammari2023spectralconvergencelargefinite}. This furnished a plethora of breakthroughs in the development of new nanotechnologies \cite{2016NatRM_Alu,Ma_Sheng_2016} and radiotherapy \cite{mdpi_ostyn}. Previous work focused on achieving localisation through the introduction of a defect in space, such as the geometry or the wave speed inside a specific resonator, which led to the emergence of space localisation. However, only very few research papers have been devoted to the derivation of time localisation \cite{THOMES_spacetimeloc}. In this paper, we will demonstrate the existence of time localised waves upon the introduction of a defect in time. Our ultimate goal is to combine defects in space and time to achieve space-time localised waves. Here, space-time localisation refers to a transition from a spatial delocalistion to spatial localisation, followed by a transition back to delocalisation.  We shall show numerically that the introduction of a defect in space and time provokes the emergence of space-time localised waves.\par
	Our approach relies on the \textit{capacitance matrix} for one dimensional infinite systems of high-contrast subwavelength resonators \cite{jinghao-silvio2023}, we derive an ordinary differential equation (ODE) formulation for the subwavelength resonant frequencies from first principles under the assumption that the material parameters are time-dependent. These formulas set the ground for an efficient numerical computation of the subwavelength resonant frequencies and the eigenmodes. We introduce the \textit{degree of localisation} as a quantity determining whether an eigenmode is localised. \par 
	This paper is structured as follows. We start by providing an overview of the geometrical setup of the one dimensional metamaterial and the definition of its time-dependent material parameters in Section \ref{sec:math_setup}. Opposed to \cite{jinghao_liora,ammari2024scattering} we assume the sources to be located on the resonator's boundaries instead of inside the resonators, which requires the derivation of new approximation formulas in the form of ODEs involving the so-called capacitance matrix in Section \ref{sec:cap_approx}. We then prove capacitance approximation formulas for time, space and space-time localisation in Section \ref{sec:loc_formulas}. In Section \ref{sec:numerical_exp} we show the hybridisation of the eigenmodes through the introduction of a defect in space or time. We also show the existence of space-time localised modes numerically. Lastly, we summarise our results in Section \ref{sec:concl}. 
	
	\section{Mathematical setup in the periodic structure}\label{sec:math_setup}
	We consider an infinite one dimensional metamaterial composed of an infinite array of the unit cell $Y:=(0,L)$, for some $L>0$, containing $N$ subwavelength resonators, sketched in \Cref{fig:setup}. In this setting, a resonator is an interval with time-modulated material parameters, embedded in a static background. We denote the resonators within the unit cell $Y$ by $D_i:=(x_i^-,x_i^+)$, for all $i=1,\dots,N$, each of length $\ell_i:=x_i^+-x_i^-$ and spacing $\ell_{i(i+1)}:=x_{i+1}^--x_i^+$ between $D_i$ and $D_{i+1}$. The strictly monotonically increasing series $(x_i^{\pm})_{1\leq i\leq N}$ defines the $2N$ boundary points of the resonators inside $Y$. This sequence can be extended to an infinite sequence $(x_i^{\pm})_{i\in\mathbb{Z}}$ defined by $x_{i+N}^{\pm}:=x_i^{\pm}+L$. In what follows, we denote by $D$ the disjoint union of resonators $D_i$, for $i=1,\dots,N$, inside $Y$. This allows us to introduce the notion of a lattice $\Lambda\subset\mathbb{R}$, defined as
	\begin{align}
		\Lambda:=\{mL:m\in\mathbb{Z}\}.
	\end{align}
	We index the cells by the integer $m\in\mathbb{Z}$ and the resonators inside the $m$-th cell by $D_i^m$, $\forall\,i=1,\dots,N$. Note that when we omit the superscript $m$, we refer to the central cell indexed by $m=0$. We refer to Figure \ref{fig:setup} for an illustration of the geometrical setup of the material.
	\begin{figure}[H]
		\centering
		\includegraphics[width=0.7\textwidth]{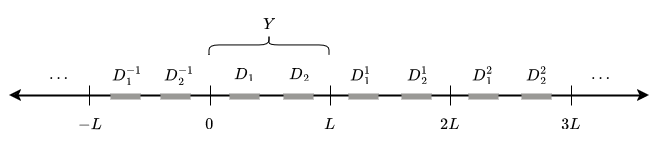}
		\caption{An illustration of the setup of an infinite material with $N=2$ resonators in the unit cell $Y$.}
		\label{fig:setup}
	\end{figure}\par 
	We suppose that there are two bulk material parameters $\kappa$ and $\rho$, given by
	\begin{align}\label{eq:kandrho}
		\kappa(x,t)=\begin{cases} 
			\kappa_0, & x \notin {D}; \\
			\kappa_{\mathrm{r}}^{i} \kappa_i(t), & x \in D_i^m,
		\end{cases}\quad \rho(x)=\begin{cases} 
			\rho_0, & x \notin {D}; \\
			\rho^i_{\mathrm{r}}, & x \in D_i^m,
		\end{cases}
	\end{align}
	for all $m\in\mathbb{Z}$. Note that if we omit the superscript $i$ in $\rho_{\mathrm{r}}^i$ and $\kappa_{\mathrm{r}}^i$, we assume that it has the same value for each $i=1,\dots,N$. Note also that we assume $\rho$ to be constant in $t$. In order to achieve spatio-temporal localisation, we additionally assume that the resonators have time-dependent boundary-layer sources $\srs_i(t)$ which enter the transmission boundary conditions across the boundary of $D$. We assume that $\kappa_i(t)$ and $\srs_i(t)$ are periodic functions with frequency $\Omega$ (and periodicity $T=2\pi/\Omega$). Furthermore, assume that their inverse have finite Fourier expansions given by
	\begin{align}\label{def:invrk_fourier}
		\frac{1}{\srs_i(t)}=\sum\limits_{n=-M}^M\srsf_{i,n}\mathrm{e}^{\mathrm{i}n\Omega t},\quad \frac{1}{\kappa_i(t)}=\sum\limits_{n=-M}^Mk_{i,n}\mathrm{e}^{\mathrm{i}n\Omega t}
	\end{align}
	for some $M\in\mathbb{N}$ such that $M=O\left(\delta^{-\gamma/2}\right)$, for some $\gamma\in(0,1)$ \cite{erik_JCP}.\par
	
	Let the contrast parameter $\delta$ and the wave speeds $v_0, v_{\mathrm{r}}$ be given by
	\begin{align}
		\delta:=\frac{\rho_{\mathrm{r}}}{\rho_0},\quad v_0:=\sqrt{\frac{\kappa_0}{\rho_0}},\quad v_{\mathrm{r}}:=\sqrt{\frac{\kappa_{\mathrm{r}}}{\rho_{\mathrm{r}}}},
	\end{align}
	respectively. Note that we can allow the wave speed $v_{\mathrm{r}}$ inside each resonator to differ. To attain subwavelength resonances, we enforce $\delta\ll1$, \textit{i.e.}, we consider that the subwavelength resonators are highly contrasting inclusions \cite{mcmpp}.\par 
	Let $u(x,t)$ be the wave field supported in the low-frequency regime and such that $u(x,t)\mathrm{e}^{-\mathrm{i}\omega t}$ is $T$-periodic with respect to time $t$ such that:
	\begin{equation}\label{eq:dD_system_u} 
		\left\{
		\begin{array} {ll}
			\ds \frac{\p^2}{\p t^2}u(x,t) - v_0^2\frac{\p^2}{\p x^2} u(x,t) = 0, \quad &x\notin D, \\[1em]
			\ds \frac{\p }{\p t } \frac{1}{\kappa_i(t)} \frac{\p}{\p t}u(x,t) - v_{\mathrm{r}}^2 \frac{\p^2}{\p x^2} u(x,t) = 0, \quad &x\in D_i, \quad i=1,\ldots,N,\\[1em]
			\ds u|_{\pm}(x_i^{\pm},t) = \frac{1}{\srs_i(t)}u|_{\mp}(x_i^{\pm},t), \quad &i=1,\dots,N, \\[1em] 	
			\ds \delta\frac{\p u}{\p x }\bigg|_{\pm}(x_i^{\pm},t) = \frac{1}{\srs_i(t)}\ds \frac{\p u}{\p x }\bigg|_{\mp}(x_i^{\pm},t), &i=1,\dots,N, \\[1em]
			u \text{ is an outgoing wave,}
		\end{array}		
		\right.
	\end{equation}
	where we use the notation
	\begin{align}
		\left.w\right|_{\pm}(x):=\lim_{h\rightarrow0,\,h>0}w(x\pm h).
	\end{align}\par 
	For a non-defected structure the wave field $u$ has the Fourier expansion
	\begin{align}
		u(x,t)=\mathrm{e}^{\mathrm{i}\omega t}\sum\limits_{n=-\infty}^{\infty}v_n(x)\mathrm{e}^{\mathrm{i}n\Omega t}.
	\end{align}
	Upon decomposing the $L^2$-functions $v_n(x)$ into their superposition of Bloch waves $\Hat{v}_n(x,\alpha)$ with momentum $\alpha\in Y^*:=(-\pi/L,\pi/L]$ \cite{jinghao_liora} and substituting this into \eqref{eq:dD_system_u}, we arrive at a boundary-value problem posed on the modes $v_n(x,\alpha):=\Hat{v}_n(x,\alpha)\mathrm{e}^{\mathrm{i}\alpha x}$ as follows:
	\begin{equation}\label{eq:dD_system_vn} 
		\left\{
		\begin{array} {ll}
			\frac{\mathrm{d}^2}{\mathrm{d}x^2}v_n+\left(k^{(n)}\right)^2v_n = 0, \quad &x\notin D, \\[1em]
			\frac{\mathrm{d}^2}{\mathrm{d}x^2}v_n+\left(k_{\mathrm{r}}^{(i,n)}\right)^2v_{i,n}^{**} = 0, \quad &x\in D_i, \quad i=1,\ldots,N,\\[1em]
			\ds v_n|_{\pm}(x_i^{\pm},\alpha) = v_{i,n}^*|_{\mp}(x_i^{\pm},\alpha), \quad &i=1,\dots,N, \\[1em] 	\ds \delta\frac{\p v_n}{\p x }\bigg|_{\pm}(x_i^{\pm},\alpha) = \frac{\p v_{i,n}^*}{\p x }\bigg|_{\mp}(x_i^{\pm},\alpha) , &i=1,\dots,N, 
		\end{array}		
		\right.
	\end{equation}
	where 
	\begin{equation}\label{def:conv_v}
		v_{i,n}^*(x,\alpha)=\sum_{m=-\infty}^{\infty} \srsf_{i,m} v_{n-m}(x,\alpha), \quad v_{i,n}^{* *}(x,\alpha)=\frac{1}{\omega+n \Omega} \sum_{m=-\infty}^{\infty} k_{i,m}(\omega+(n-m) \Omega) v_{n-m}(x,\alpha).
	\end{equation}
	Here, $\srsf_{i,n}$ and $k_{i,n}$ are the Fourier coefficients defined in \eqref{def:invrk_fourier}, and the $n$-th wave numbers outside and inside the resonators are given by
	\begin{equation}
		k^{(n)}:=\frac{\omega+n\Omega}{v_0},\quad k_{\mathrm{r}}^{(i,n)}:=\frac{\omega+n\Omega}{v_{\mathrm{r}}^{i}},
	\end{equation}
	respectively.
	
	\begin{rem}
		For the numerical experiments that follow, we assume the time-dependent material parameters $\kappa_i(t)$ and $\srs_i(t)$ to be given by 
		\begin{align}\label{def:rho_kappa}
			\kappa_i(t):=\frac{1}{1+\varepsilon_{\kappa}\cos\left(\Omega t+ \phi_{\kappa}^i\right)},\quad \srs_i(t):=\frac{1}{1+\varepsilon_{\srs}\cos\left(\Omega t+ \phi_{\srs}^i\right)},
		\end{align}
		with 
		\begin{align}
			\phi_{\kappa}^i=\phi_{\srs}^i=\begin{cases}
				\frac{\pi}{i-1},&i>0;\\
				0,&i=0.
			\end{cases}
		\end{align}
		In this case, the non-zero Fourier coefficients of their inverse are given by
		\begin{align}
			k_{i,-1}=\frac{\varepsilon_{\kappa}\mathrm{e}^{-\mathrm{i}\pi\phi^i_{\kappa}}}{2},\quad k_{i,0}=1,\quad k_{i,1}=\frac{\varepsilon_{\kappa}\mathrm{e}^{\mathrm{i}\pi\phi^i_{\kappa}}}{2},\\
			\srsf_{i,-1}=\frac{\varepsilon_{\srs}\mathrm{e}^{-\mathrm{i}\pi\phi^i_{\srs}}}{2},\quad \srsf_{i,0}=1,\quad \srsf_{i,1}=\frac{\varepsilon_{\srs}\mathrm{e}^{\mathrm{i}\pi\phi^i_{\srs}}}{2}.
		\end{align}
	\end{rem}
	\subsection{Physical Interpretation}
	Opposed to the governing equations used in previous works \cite{jinghao_liora,ammari2024scattering}, we assume that $\rho$ is constant in $t$ and we instead have sources $\srs$ located on the resonator's boundaries. This alters the transmission conditions posed on the boundaries $x_i^{\pm}$, $\forall\,i=1,\dots,N$, as in \eqref{eq:dD_system_u}. Our previously derived capacitance matrix approximation formulas based on \cite[Eq. (13)]{jinghao_liora} made it apparent that $\rho$ does not affect the band functions at leading order in $\delta$. However, with the newly introduced governing equations \eqref{eq:dD_system_u} we shall prove in Section \ref{sec:cap_approx} that both $\srs$ and $\kappa$ impact the band functions at leading order in $\delta$. This is desirable as it allows the formation of simultaneous frequency band gaps and momentum gaps.
	
	\subsection{Capacitance matrix approximation}\label{sec:cap_approx}
	We introduce the following operator \cite{jinghao_liora}:
	\begin{definition}[Operator $I_{\partial D}$]
		For simplicity of notation, for any smooth function $f:\mathbb{R}\rightarrow \mathbb{R}$, we define 
		\begin{equation}
			I_{\partial D_i}[f]:= \frac{\mathrm{d} f}{\mathrm{d} x}\bigg\vert_-(x_i^-)- \frac{\mathrm{d} f}{\mathrm{d} x}\bigg\vert_+(x_i^+).
		\end{equation}
	\end{definition}
	The following marks a key result for the derivation of the capacitance matrix approximation:
	\begin{lemma}\label{lemma:vn_const}
		As $\delta\to0$, $v_{n}(x,\alpha)\big\rvert_{(x_i^-,x_i^+)}=v_{i,n}+O(\delta^{(1-\gamma)/2})$, for fixed $\alpha\in Y^*$.
	\end{lemma}
	\begin{proof}
		See \cite[Lemma 4.1]{erik_JCP}.
	\end{proof}
	Moreover, we use the following expression for $v_n$:
	\begin{lemma}\label{lemma:vn_sum}
		As $\delta\to0$,
		\begin{align}
			v_n(x)=\sum\limits_{m=-\infty}^{\infty}\sum\limits_{j=1}^{N}V_i^{\alpha}(x)\srsf_{j,m}v_{j,n-m} +O(\delta^{(1-\gamma)/2}),\quad\forall\,x\in D_i,
		\end{align}
		where the functions $V_i^\alpha:\mathbb{R}\rightarrow \mathbb{R}$ are solutions to the following equations:
		\begin{equation}\label{def:Vialpha}
			\begin{cases}-\frac{\mathrm{d}^2}{\mathrm{d}x^2} V_i^\alpha=0, & x\in(0,L) \backslash D, \\ V_i^\alpha(x)=\delta_{i j}, & x \in D_j, \\ V_i^\alpha(x+m L)=\mathrm{e}^{\mathrm{i} \alpha m L} V_i^\alpha(x), & m \in \mathbb{Z}.\end{cases}
		\end{equation} 
	\end{lemma}
	\begin{proof}
		The result is an immediate generalisation of \cite[Lemma 5.2]{jinghao_liora} and its proof.
	\end{proof}
	Applying the operator $I_{\partial D_i}$ to $v_n$ and using the transmission condition at the boundaries allows us to arrive at
	\begin{align}\label{eq:Ivn_**}
		I_{\partial D_i}[v_n]&=-\frac{1}{\delta}\int_{x_i^-}^{x_i^+}\frac{\mathrm{d}^2}{\mathrm{d}x^2}v_{i,n}^*(x,\alpha)\,\mathrm{d}x\nonumber\\
		&=-\frac{1}{\delta}\sum\limits_{m=-\infty}^{\infty}\srsf_{i,m}\int_{x_i^-}^{x_i^+}\frac{\mathrm{d}^2}{\mathrm{d}x^2}v_{n-m}(x,\alpha)\,\mathrm{d}x\nonumber\\
		&=\frac{\rho_{\mathrm{r}}}{\delta\kappa_{\mathrm{r}}}\sum\limits_{m=-\infty}^{\infty}\srsf_{i,m}(\omega+(n-m)\Omega)^2\int_{x_i^-}^{x_i^+}v_{i,n-m}^{**}(x,\alpha)\,\mathrm{d}x.
	\end{align}
	On the other hand, Lemma \ref{lemma:vn_sum} leads to the introduction of the capacitance matrix coefficients as follows:
	\begin{align}\label{eq:Ivn_cap}
		I_{\partial D_i}[v_n]=\sum\limits_{m=-\infty}^{\infty}\sum\limits_{j=1}^N\srsf_{j,m}v_{j,n-m}C_{ij}^{\alpha},
	\end{align}
	where the capacitance matrix is given by \cite{jinghao-silvio2023}
	\begin{align}
		C^{\alpha}:=\left(C_{ij}^{\alpha}\right)_{i,j=1,\dots,N},\quad C_{ij}^{\alpha}:=I_{\partial D_i}[V_j^{\alpha}].
	\end{align}
	Equating \eqref{eq:Ivn_cap} and \eqref{eq:Ivn_**}, we conclude that
	\begin{align}\label{eq:pre_capapprox}
		\sum\limits_{m=-\infty}^{\infty}\sum\limits_{j=1}^N\srsf_{j,m}v_{j,n-m}C_{ij}^{\alpha}=\frac{\rho_{\mathrm{r}}}{\delta\kappa_{\mathrm{r}}}\sum\limits_{l=-\infty}^{\infty}\srsf_{i,n-l}(\omega+l\Omega)\sum\limits_{m=-\infty}^{\infty}k_{i,l-m}(\omega+m\Omega)\int_{x_i^-}^{x_i^+}v_l(x,\alpha)\,\mathrm{d}x,
	\end{align}
	for fixed $\alpha\in Y^*$ and $n\in\mathbb{Z}$. Next, we use that for fixed $\alpha$
	\begin{align}
		\int_{x_i^-}^{x_i^+}v_n(x,\alpha)\,\mathrm{d}x\approx\ell_iv_{i,n}.
	\end{align}
	By taking the inverse transform of \eqref{eq:pre_capapprox} we get
	\begin{align}\label{eq:ODEperiodic}
		\sum\limits_{j=1}^NC_{ij}^{\alpha}w_j(t)=-\frac{\ell_i\rho_{\mathrm{r}}}{\delta\kappa_{\mathrm{r}}\srs_i(t)}\frac{\mathrm{d}}{\mathrm{d}t}\left(\frac{1}{\kappa_i(t)}\frac{\mathrm{d}}{\mathrm{d}t}\srs_i(t)w_i(t)\right),
	\end{align}
	where 
	\begin{align}
		w_i(t):=\frac{c_i(t)}{\srs_i(t)},\quad c_i(t):=\mathrm{e}^{\mathrm{i}\omega t}\sum\limits_{n=-\infty}^{\infty}v_{i,n}\mathrm{e}^{\mathrm{i}n\Omega t}.
	\end{align}
	This proves the following result:
	\begin{thm}
		As $\delta\to0$, the quasifrequencies in the subwavelength regime are, at leading order, given by the quasifrequencies of the ODE
		\begin{align}\label{eq:capmat_ode}
			M^{\alpha}(t)\Psi(t)+\Psi^{''}(t)=0,
		\end{align}
		where $M^{\alpha}(t)=\frac{\delta\kappa_{\mathrm{r}}}{\rho_{\mathrm{r}}}W_1(t)C^{\alpha}W_2(t)+W_3(t)$ with $W_1,\,W_2,\,W_3$ being diagonal matrices defined as
		\begin{align}
			(W_1)_{ii}=\frac{\sqrt{\kappa_i}\srs_i}{\ell_i},\quad (W_2)_{ii}=\frac{\sqrt{\kappa_i}}{\srs_i},\quad (W_3)_{ii}=\frac{\sqrt{\kappa_i}}{2}\frac{\mathrm{d}}{\mathrm{d}t}\frac{\kappa_i^{'}}{\kappa_i^{3/2}},
		\end{align}
		for $i=1,\dots,N$.
	\end{thm}
	
	\subsection{Real-space capacitance formulation}\label{sec:realspace}
	The results of the previous section, which were posed in terms of the dual (Floquet-Bloch) variable $\alpha\in Y^*$, can be rephrased in terms of the real-space variable $m\in \Lambda$. We introduce the one dimensional Floquet-Bloch transform $\mathcal{I}:\left(\ell^2\left(\Lambda\right)\right)^N\to\left(L^2\left(Y^*\right)\right)^N$ and its inverse $\mathcal{I}^{-1}:\left(L^2\left(Y^*\right)\right)^N\to\left(\ell^2\left(\Lambda\right)\right)^N$, which are given by
	\begin{align}
		\mathcal{I}[f](\alpha):=\sum\limits_{m\in\Lambda}f(m)\mathrm{e}^{\mathrm{i}\alpha m},\qquad \mathcal{I}^{-1}[F](m):=\frac{L}{2\pi}\int_{-\pi/L}^{\pi/L}F(\alpha)\mathrm{e}^{-\mathrm{i}\alpha m}\,\mathrm{d}\alpha,
	\end{align}
	respectively. The following multiplication-rule for the inverse Floquet-Bloch transform is a basic application of the convolution theorem; for completeness, we include the proof as the result will be applied.
	\begin{lemma}\label{lemma:multip}
		Let $F,G\in \left(L^2\left(Y^*\right)\right)^N$. Then,
		\begin{align}
			\mathcal{I}^{-1}[FG](m)=\sum\limits_{n\in\Lambda}\mathcal{I}^{-1}[F](n)\mathcal{I}^{-1}[G](m-n).
		\end{align}
	\end{lemma}
	\begin{proof}
		Firstly, note that we can write
		\begin{align}
			F(\alpha)=\sum\limits_{n\in\Lambda}\mathcal{I}^{-1}[F](n)\mathrm{e}^{\mathrm{i}\alpha n}.
		\end{align}
		Inserting the above expression into the definition of $\mathcal{I}^{-1}$ allows us to conclude the following:
		\begin{align}
			\mathcal{I}^{-1}[FG](m)&=\frac{L}{2\pi}\int_{-\pi/L}^{\pi/L}\sum\limits_{n\in\Lambda}\mathcal{I}^{-1}[F](n)\mathrm{e}^{\mathrm{i}\alpha n}G(\alpha)\mathrm{e}^{-\mathrm{i}\alpha m}\,\mathrm{d}\alpha\nonumber\\
			&=\sum\limits_{n\in\Lambda}\mathcal{I}^{-1}[F](n)\frac{L}{2\pi}\int_{-\pi/L}^{\pi/L}G(\alpha)\mathrm{e}^{-\mathrm{i}\alpha(m-n)}\,\mathrm{d}\alpha\nonumber\\
			&=\sum\limits_{n\in\Lambda}\mathcal{I}^{-1}[F](n)\mathcal{I}^{-1}[G](m-n),
		\end{align}
		which proves the result.
	\end{proof}
	Based on the quasiperiodic capacitance matrix $C^\alpha, \alpha \in Y^*$, we now define the real-space capacitance matrix $C^m, \ m \in \Z$ as $C^m = \mathcal{I}^{-1}[C^\alpha](m)$.
	
	\begin{lemma}\label{lemma:Cm_realspace}
		The real-space capacitance matrix $C^m, m \in \Z,$ corresponding to periodically reoccurring $N$ resonators is given by 
		\begin{align}
			\left(C^m\right)_{ij}=\begin{cases}
				\left(\frac{1}{\ell_{(i-1)i}}+\frac{1}{\ell_{i(i+1)}}\right)\frac{L}{m\pi}\sin\left(\frac{\pi m}{L}\right),&i=j;\\[1em]
				-\frac{1}{\ell_{ij}}\frac{L}{m\pi}\sin\left(\frac{\pi m}{L}\right),&j=i+1;\\[1em]
				-\frac{1}{\ell_{ji}}\frac{L}{m\pi}\sin\left(\frac{\pi m}{L}\right),&i=j+1;\\[1em]
				-\frac{1}{\ell_{N(N+1)}}\frac{L}{\pi(m+L)}\sin\left(\frac{\pi(m+L)}{L}\right),&i=1,j=N;\\[1em]
				-\frac{1}{\ell_{N(N+1)}}\frac{L}{\pi(m-L)}\sin\left(\frac{\pi(m-L)}{L}\right),&i=N,j=1.
			\end{cases}
		\end{align}
	\end{lemma}
	\begin{proof}
		The definition of $C^m$ can be achieved by applying the inverse Floquet-Bloch transform $\mathcal{I}^{-1}$ to $C^{\alpha}$.
	\end{proof}
	In the static case, the subwavelength resonant frequencies $\omega_i^{\alpha}$ and the corresponding eigenmodes $\mathbf{v}_i^{\alpha}$ are given by the eigenvalues and eigenvector pairs of
	\begin{align}
		\frac{\delta\kappa_{\mathrm{r}}}{\rho_{\mathrm{r}}}\mathcal{L}C^{\alpha}\mathbf{v}=\omega \mathbf{v},
	\end{align}
	where $\mathcal{L}:=\mathrm{diag}\left(1/\ell_i\right)_{i=1,\dots,N}$.
	\begin{definition}[Generalised Capacitance Matrix]
		We define the generalised capacitance matrix to be given by the following matrix product: \begin{align}
			\mathcal{C}^{\alpha}:=\frac{\delta\kappa_{\mathrm{r}}}{\rho_{\mathrm{r}}}\mathcal{L}C^{\alpha}.
		\end{align}
	\end{definition}
	Before proceeding with the formulation of a real--space capacitance matrix approximation, we must introduce some new notation. We define the column-vectors $\mathbf{v}^m_n:=\left[v_{i,n}^m\right]_{i=1}^N\in\mathbb{R}^N$, and we define the diagonal matrices 
	\begin{align}
		\Srs(t):=\mathrm{diag}\left(\srs_i(t)\right)_{i=1}^N,\quad K(t)=\mathrm{diag}\left(\kappa_i(t)\right)_{i=1}^N.
	\end{align}
	Equation \eqref{eq:ODEperiodic} can now be equivalently expressed as a vector ODE with the infinite matrices
	\begin{align}
		\mathcal{C}:=\begin{bsmallmatrix}\sddots & \svdots & \svdots & \svdots & \svdots & \sadots \\
			\cdots & \mathcal{C}^0 & \mathcal{C}^1 & \mathcal{C}^2 & \mathcal{C}^3 & \cdots \\
			\cdots & \mathcal{C}^{-1} & \mathcal{C}^0 & \mathcal{C}^1 & \mathcal{C}^2 & \cdots \\
			\cdots & \mathcal{C}^{-2} & \mathcal{C}^{-1} & \mathcal{C}^0 & \mathcal{C}^1 & \cdots \\
			\cdots & \mathcal{C}^{-3} & \mathcal{C}^{-2} & \mathcal{C}^{-1} & \mathcal{C}^0 & \cdots \\
			\sadots & \svdots & \svdots & \svdots & \svdots & \sddots\end{bsmallmatrix}, \quad	\Srsc(t):=\begin{bsmallmatrix}
			\sddots & 0 & & \\
			0 & \Srs(t) & 0 & \\
			& 0 & \sddots &
		\end{bsmallmatrix},& \quad \mathcal{K}(t):=\begin{bsmallmatrix}
			\sddots & 0 & & \\
			0 & K(t) & 0 & \\
			& 0 & \sddots &
		\end{bsmallmatrix},
	\end{align}
	and infinite vector
	\begin{align}
		\mathbf{v}(t):=\begin{bsmallmatrix}
			\\\svdots \\ \mathbf{v}^{-1}(t) \\ \mathbf{v}^{0}(t) \\ \mathbf{v}^{1}(t) \\ \svdots\\
		\end{bsmallmatrix},\quad \mathbf{v}^m(t):=\sum\limits_{k\in\mathbb{Z}}\mathbf{v}^m_k\mathrm{e}^{\mathrm{i}(\omega+k\Omega)t}
	\end{align}
	as follows:
	\begin{align}\label{eq:v_ODE_perioidc}
		\mathcal{C}^{\alpha}\Srsc^{-1}(t)\mathbf{v}(t)=\Srsc^{-1}(t)\frac{\mathrm{d}}{\mathrm{d}t}\left(\mathcal{K}^{-1}(t)\frac{\mathrm{d}}{\mathrm{d}t}\mathbf{v}(t)\right).
	\end{align}
	As we shall see, this real-space representation will be advantageous when treating material defects in \Cref{sec:loc_formulas}.
	
	\subsection{Band gaps and momentum gaps}
	We seek to analyse the so-called band structure of the material, which describes the quasifrequency-to-momentum relationship of the propagating waves \cite{jinghao2}.
	\begin{definition}[Band Gap]
		A band gap is defined to be the regime of subwavelength frequencies $\omega$ with which waves are unable to propagate through the medium, but decay exponentially instead \cite{jinghao2}. Mathematically speaking, this means that, for some $1\leq i\leq N-1$,
		\begin{align}
			\max\limits_{\alpha\in Y^*}\omega_i^{\alpha} < \min\limits_{\alpha\in Y^*}\omega_{i+1}^{\alpha}.
		\end{align}
	\end{definition}
	\begin{definition}[Momentum Gap]
		A momentum gap is defined to be a band gap in the momentum variable $\alpha \in Y^*$ \cite{erik_JCP}.
	\end{definition}
	Our numerical experiments show that modulating $\kappa$ in time provokes the formation of momentum gaps and modulating $\srs$ in time provokes the formation of band gaps. In Figure \ref{fig:bandfct_kbgap} we show that by modulating $\srs$ and $\kappa$ in time, we can achieve the formation of band and momentum gaps at the same time.
	\begin{figure}[H]
		\centering
		\includegraphics[width=0.58\textwidth]{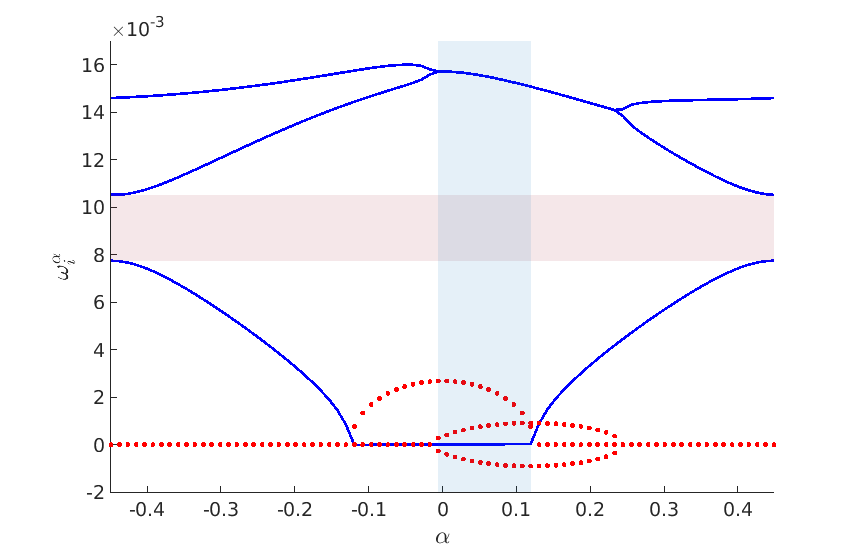}
		\caption{The band structure corresponding to $N=3$ resonators with the following parameter values: $\ell_1=\ell_2=\ell_3=1,\,\ell_{12}=\ell_{34}=1,\,\ell_{23}=2,\,v_0=v_{\mathrm{r}}=1,\,\Omega=0.034,\delta=0.0001,\,\varepsilon_{\kappa}=0.4,\,\varepsilon_{\srs}=0.2$. The blue line marks the real parts and the red dots the imaginary parts of the subwavelength resonant frequencies. The blue shaded area marks a momentum gap and the red shaded area marks a band gap.}
		\label{fig:bandfct_kbgap}
	\end{figure}
	
	\section{Localisation in structures with defects}\label{sec:loc_formulas}
	In \Cref{sec:math_setup}, we saw the presence of simultaneous band gaps and momentum gaps in structures which are periodic in both space and time. The band structure is concisely described by the capacitance matrix ODE \eqref{eq:v_ODE_perioidc}. The goal, now, is to achieve space-time localisation by breaking this periodicity. In this section, we will consider local material defects in both space and time, and derive analogous capacitance formulations.
	
	\subsection{Space localisation}
	To achieve space localisation we assume $\rho$ to be inhomogeneous in space. We remark that it is redundant to consider $\kappa$ to be spatially inhomogeneous because this leads to the same equations as the case of a spatial defect in $\rho$.
    We index the cells by $m\in\mathbb{Z}$ and assume that the time-modulated parameter $\rho$ does not only differ between each resonator, but also between cells, that is:
		\begin{align}\label{eq:rho_nonper}
			\rho(x)=\begin{cases}
				\rho,&x\notin D;\\
				\rho_{\mathrm{r}}^{m,i},&x\in D_i^m.
			\end{cases}
		\end{align}
        
	We generalise \cite[Proposition 4.2]{ammarihiltunen2020edge} to the  time-dependent case:
	\begin{prop}\label{prop:cap_vmi}
		Any localised solution $u$ to \eqref{eq:dD_system_u} corresponding to a subwavelength quasifrequency $\omega$ satisfies the following ODE:
		\begin{multline}\label{eq:cap_approx_vmi}
			\sum\limits_{j=1}^NC_{ij}^{\alpha}\left(\frac{1}{\srs_i(t)}\sum\limits_{n\in\mathbb{Z}}\left(\sum\limits_{m\in\mathbb{Z}}v_{i,n}^m\mathrm{e}^{\mathrm{i}\alpha mL}\right)\mathrm{e}^{\mathrm{i}(\omega+n\Omega)t}\right) \\
			=\frac{\ell_i}{\srs_i(t)\kappa_{\mathrm{r}}}\frac{\mathrm{d}}{ \mathrm{d}t}\left(\frac{1}{\kappa_i(t)}\frac{\mathrm{d}}{\mathrm{d}t}\left(\sum\limits_{n\in\mathbb{Z}}\left(\sum\limits_{m\in\mathbb{Z}}v_{i,n}^m\mathrm{e}^{ \mathrm{i}\alpha mL}\frac{\rho_{\mathrm{r}}^{m,j}}{\delta^m}\right)\mathrm{e}^{\mathrm{i}(\omega+n\Omega)t}\right)\right),
		\end{multline}
		for all $i=1,\dots,N$.
	\end{prop}
	\begin{proof}
		We consider a cell containing $N$ resonators in which the modes $v_n$ satisfy
		\begin{equation}\label{eq:dD_system_vn_nonper} 
			\left\{
			\begin{array} {ll}
				\frac{\mathrm{d}^2}{\mathrm{d}x^2}v_n+\left(k^{(n)}\right)^2v_n = 0, \quad &x\notin D, \\[1em]
				\frac{\mathrm{d}^2}{\mathrm{d}x^2}v_n+\left(k_{\mathrm{r}}^{(i,n)}\right)^2v_{i,n}^{**} = 0, \quad &x\in D_i, \quad i=1,\ldots,N,\\[1em]
				\ds v_n|_{\pm}(x_i^{\pm}) = v_{i,n}^*|_{\mp}(x_i^{\pm}), \quad &i=1,\dots,N, \\[1em] 	
				\ds \frac{\mathrm{d} v_{i,n}^*}{\mathrm{d} x }\bigg|_{\mp}(x_i^{\pm}) = \delta\frac{\mathrm{d} v_n}{\mathrm{d} x }\bigg|_{\pm}(x_i^{\pm}), &i=1,\dots,N. 
			\end{array}		
			\right.
		\end{equation}
		Assume that the structure we are interested in is given by infinitely many adjacent cells, each containing a collection of $N$ resonators. 
        
		In view of \eqref{eq:rho_nonper},  applying the Floquet-Bloch transform to the above equations, we obtain the following equations governing $v_n^{\alpha}(x):=\mathcal{I}[v_n](x,\alpha)$:
		\begin{equation}\label{eq:dD_system_vna_nonper} 
			\left\{
			\begin{array} {ll}
				\frac{\mathrm{d}^2}{\mathrm{d}x^2}v^{\alpha}_n+\left(k^{(n)}\right)^2v_n^{\alpha} = 0, \quad &x\notin D, \\[1em]
				\frac{\mathrm{d}^2}{\mathrm{d}x^2}v^{\alpha}_n+\frac{\left(\omega+n\Omega\right)}{\kappa_{\mathrm{r}}}\sum\limits_{m\in\mathbb{Z}}\sum\limits_{l\in\mathbb{Z}}k_{i,l}\mathrm{e}^{\mathrm{i}\alpha mL}\rho_{\mathrm{r}}^{m,i}\left(\omega+(n-l)\Omega\right)v_{n-l} = 0, \quad &x\in D_i,\\[1em]
				\ds v^{\alpha}_n|_{\pm}(x_i^{\pm}) = v_{i,n}^{*,\alpha}|_{\mp}(x_i^{\pm}), \quad &i=1,\dots,N, \\[1em] 	
				\ds \frac{\mathrm{d} v_{i,n}^{*,\alpha}}{\mathrm{d} x }\bigg|_{\mp}(x_i^{\pm}) = \sum\limits_{m=-\infty}^{\infty}\delta^m\mathrm{e}^{\mathrm{i}\alpha mL}\frac{\mathrm{d} v^{\alpha}_n}{\mathrm{d} x }\bigg|_{\pm}(x_i^{\pm}), &i=1,\dots,N. 
			\end{array}		
			\right.
		\end{equation}
		Now, apply the operator $I_{\partial D_j}$ to $v_n^{\alpha}$ and use the transmission condition and the interior equation, which results in
		\begin{align}
			I_{\partial D_j}[v_n^{\alpha}]&=-\frac{1}{\delta^m}\int_{x_j^-}^{x_j^+}\frac{\mathrm{d}^2}{\mathrm{d}x^2}v_{j,n}^{*,\alpha}\,\mathrm{d}x\nonumber\\
			&=\sum\limits_{s\in\mathbb{Z}}\srsf_{j,s}\int_{x_j^-}^{x_j^+}\frac{(\omega+(n-s)\Omega)}{\kappa_{\mathrm{r}}}\sum\limits_{m\in\mathbb{Z}}\mathrm{e}^{\mathrm{i}\alpha mL}\frac{\rho_{\mathrm{r}}^{m,j}}{\delta^m}\sum\limits_{l\in\mathbb{Z}}k_{j,l}(\omega+(n-s-l)\Omega)v^{\alpha}_{n-s-l}\,\mathrm{d}x\nonumber\\
			&=\frac{1}{\kappa_{\mathrm{r}}}\sum\limits_{s\in\mathbb{Z}}\srsf_{j,n-s}(\omega+s\Omega)\sum\limits_{m\in\mathbb{Z}}\mathrm{e}^{\mathrm{i}\alpha mL}\frac{\rho_{\mathrm{r}}^{m,j}}{\delta^m}\sum\limits_{l\in\mathbb{Z}}k_{j,s-l}(\omega+l\Omega)\int_{x_j^-}^{x_j^+}v_l^{\alpha}\,\mathrm{d}x\nonumber\\
			&=\frac{\ell_j}{\kappa_{\mathrm{r}}}\sum\limits_{s\in\mathbb{Z}}\srsf_{j,n-s}(\omega+s\Omega)\sum\limits_{m\in\mathbb{Z}}\sum\limits_{l\in\mathbb{Z}}v_{j,l}^m\mathrm{e}^{\mathrm{i}\alpha mL}\frac{\rho_{\mathrm{r}}^{m,j}}{\delta^m}k_{j,s-l}(\omega+l\Omega).
		\end{align}
		In the last step we used the result of Lemma \ref{lemma:vn_const}, namely,
		\begin{align}
			v^{\alpha}_n(x)=v_{i,n}=\sum\limits_{m\in\mathbb{Z}}v_{i,n}^m\mathrm{e}^{\mathrm{i}\alpha mL},\qquad\forall\,x\in D_i.
		\end{align}
		On the other hand we may use the definition of the capacitance matrix coefficients and obtain
		\begin{align}
			I_{\partial D_j}[v_n^{\alpha}]&=I_{\partial D_j}\left[\sum\limits_{l=-M}^M\sum\limits_{i=1}^NV_i^{\alpha}(x)\srsf_{i,l}\sum\limits_{m\in\mathbb{Z}}v_{i,n-l}^m\mathrm{e}^{\mathrm{i}\alpha mL}\right]\nonumber\\
			&=\sum\limits_{l=-M}^M\sum\limits_{i=1}^N\srsf_{i,l}C_{ij}^{\alpha}\sum\limits_{m\in\mathbb{Z}}v_{i,n-l}^m\mathrm{e}^{\mathrm{i}\alpha mL},
		\end{align}
		where $V_i^{\alpha}$ is defined by \eqref{def:Vialpha}. Thus, we conclude that
		\begin{align}
			\sum\limits_{l=-M}^M\sum\limits_{i=1}^N\srsf_{i,l}C_{ij}^{\alpha}\sum\limits_{m\in\mathbb{Z}}v_{i,n}^m\mathrm{e}^{\mathrm{i}\alpha mL}=\frac{\ell_j}{\kappa_{\mathrm{r}}}\sum\limits_{s\in\mathbb{Z}}\srsf_{j,n-s}(\omega+s\Omega)\sum\limits_{m\in\mathbb{Z}}\sum\limits_{l\in\mathbb{Z}}v_{j,l}^m\mathrm{e}^{\mathrm{i}\alpha mL}\frac{\rho_{\mathrm{r}}^{m,j}}{\delta^m}k_{j,s-l}(\omega+l\Omega),
		\end{align}
		which can equivalently be written as \eqref{eq:cap_approx_vmi}.
	\end{proof}
	Following \Cref{sec:realspace}, we rephrase \eqref{eq:cap_approx_vmi} into a real-space capacitance formulation. We introduce the quantities $b_i^m$ such that
	\begin{align}\label{eq:bmi}
		\delta^m\left(v_i^m\right)^2=\delta^m\frac{\kappa_{\mathrm{r}}}{\rho_{\mathrm{r}}^{m,i}}=\delta\frac{\kappa_{\mathrm{r}}}{\rho^i_{\mathrm{r}}}b_i^m=\delta\left(v_i\right)^2b_i^m
	\end{align}
	and we define the diagonal matrices 
	\begin{align}\label{eq:B}
		B^m:=\mathrm{diag}\left(b_i^m\right)_{i=1}^N,\quad B:=\begin{bsmallmatrix}
			\sddots & \svdots & \svdots & \svdots & \svdots & \sadots \\
			\cdots & B^{-1} & 0 & 0 & 0 & \cdots \\
			\cdots & 0 & B^0 & 0 & 0 & \cdots \\
			\cdots & 0 & 0 & B^1 & 0 & \cdots \\
			\cdots & 0 & 0 & 0 & B^2 & \cdots \\
			\sadots & \svdots & \svdots & \svdots & \svdots & \sddots  
		\end{bsmallmatrix}.
	\end{align}\par 
	We shall now use the ODE proved in Proposition \ref{prop:cap_vmi} and the multiplication-rule proved in Lemma \ref{lemma:multip} to derive a capacitance matrix approximation in the real space, rather than the dual space, as in the static case in \cite[Proposition 3.5]{ammari2023anderson}.
	\begin{prop}\label{prop:conv_cap_real}
		Any localised solution to \eqref{eq:dD_system_u}, corresponding to a subwavelength frequency $\omega$ satisfies
		\begin{align}\label{eq:conv_cap_real}
			B^m\sum\limits_{n\in\Lambda}\mathcal{C}^{m-n}\left(\sum\limits_{k\in\mathbb{Z}}\mathbf{v}_k^n\mathrm{e}^{\mathrm{i}(\omega+k\Omega)t}\Srs^{-1}(t)\right)=\Srs^{-1}(t)\frac{\mathrm{d}}{\mathrm{d}t}\left(K^{-1}(t)\frac{\mathrm{d}}{\mathrm{d}t}\sum\limits_{k\in\mathbb{Z}}\mathbf{v}_k^m\mathrm{e}^{\mathrm{i}(\omega+k\Omega)t}\right),
		\end{align}
		for every $m\in\Lambda$, where $\mathcal{C}^m=\mathcal{I}^{-1}[\mathcal{C}^{\alpha}](m)$.
	\end{prop}
	\begin{proof}
		If we assume the material to be perturbed in space, the formulation obtained in Proposition \ref{prop:cap_vmi} can be rewritten as
		\begin{align}\label{eq:pre_capreal}
			\mathcal{C}^{\alpha}\Biggl(\sum\limits_{\substack{n\in\mathbb{Z} \\ m\in\mathbb{Z}}} \mathbf{v}_n^m\mathrm{e}^{\mathrm{i}\alpha mL}\mathrm{e}^{\mathrm{i}(\omega+n\Omega)t}\Srs^{-1}(t)\Biggr)=\Srs^{-1}(t)\frac{\mathrm{d}}{\mathrm{d}t}\Biggl(K^{-1}(t)\frac{\mathrm{d}}{\mathrm{d}t}\sum\limits_{\substack{n\in\mathbb{Z} \\ m\in\mathbb{Z}}} \mathbf{v}_n^m\left(B^m\right)^{-1}\mathrm{e}^{\mathrm{i}\alpha mL}\mathrm{e}^{\mathrm{i}(\omega+n\Omega)t}\Biggr),
		\end{align}
		where $\mathcal{C}^{\alpha}$ is the generalised capacitance matrix. Recall that $\mathcal{I}[\mathbf{v}_n^m]=\mathbf{v}^{\alpha}_n$. Lastly, we take the inverse Floquet-Bloch transform of \eqref{eq:pre_capreal} and use the result of Lemma \ref{lemma:multip} for the left-hand side to obtain the desired result \eqref{eq:conv_cap_real}.\par
	\end{proof}
	With $B$ as in \eqref{eq:B}, equation \eqref{eq:conv_cap_real} can more concisely be expressed as a vector ODE 
	\begin{align}\label{eq:v_ODE}
		B\mathcal{C}^{\alpha}\Srsc^{-1}(t)\mathbf{v}(t)=\Srsc^{-1}(t)\frac{\mathrm{d}}{\mathrm{d}t}\left(\mathcal{K}^{-1}(t)\frac{\mathrm{d}}{\mathrm{d}t}\mathbf{v}(t)\right).
	\end{align}
	\subsubsection{Toeplitz matrix formulation}
	In the case of compact defects, we shall follow \cite{ammari2023anderson} and rewrite the problem \eqref{eq:v_ODE} in the form of a linear system with a (finite) Toeplitz matrix. We assume that the matrices $B^m$ are the identity matrix for all $m<0$ and $m>M$, for some positive integer $M$, \textit{i.e.},
	\begin{align}\label{def:bmi}
		b_i^m:=\begin{cases}
			1,&m<0 \,\,\text{or}\,\,m>M;\\
			1+\eta_i^m,&0\leq m\leq M,
		\end{cases}
	\end{align}
	for some $\eta_i^m\in(-1,\infty)$. We introduce the diagonal matrix $H^m:=\mathrm{diag}\left(\eta_i^m\right)_{i=1}^N$. Analogously to \cite[Proposition 3.7]{ammari2023anderson} we now prove a formula for the subwavelength resonant frequencies of \eqref{eq:v_ODE} in the form of a root-finding problem. For simplicity, we phrase the result for $N=1$, although it is straightforward to generalise to a higher number of resonators per unit cell.
	\begin{prop}\label{prop:toeplitz}
		Assume that $N=1$ and $b^m_i$ are given by \eqref{def:bmi}, then $\omega_0$ is a resonant frequency of \eqref{eq:v_ODE} if and only if
		\begin{align}\label{eq:det_cond}
			\mathrm{det}\left(I-\mathcal{H}\mathcal{T}(\omega_0)\right)=0,
		\end{align}
		where $\mathcal{H}:=\mathrm{diag}\left(G^m\right)_{m=0}^M$ for
		\begin{align}
			G^m:=\begin{bsmallmatrix}
				0 & \eta^m\frac{\varepsilon_{\srs}}{2} & \eta^m & \eta^m\frac{\varepsilon_{\srs}}{2} & 0 & & \cdots & & 0 \\
				& & & & & & & & & \\
				\vdots & & \ddots & \ddots & \ddots & \ddots & \ddots & & \vdots \\
				& & & & & & & & & \\
				0 & & \cdots & & 0 & \eta^m\frac{\varepsilon_{\srs}}{2} & \eta^m & \eta^m\frac{\varepsilon_{\srs}}{2} & 0
			\end{bsmallmatrix}
		\end{align}
		and 
		\begin{align}
			\mathcal{T}(\omega):=\begin{bsmallmatrix}
				T^0 & T^1 & \cdots & T^M \\
				T^{-1} & \ddots & & T^{M-1} \\
				\vdots & & \ddots & \vdots \\
				T^{-M} & T^{-(M-1)} & \cdots & T^0
			\end{bsmallmatrix},\qquad 
			T^m(\omega):=-\frac{L}{2\pi}\int_{-\frac{\pi}{L}}^{\frac{\pi}{L}}\mathcal{C}^{\alpha}\Gamma^{\alpha}(\omega)^{-1}\mathrm{e}^{\mathrm{i}\alpha m}\,\mathrm{d}\alpha, 
		\end{align}
		with
		\begin{align}
			\Gamma^{\alpha}(\omega):=\begin{bsmallmatrix}
				\gamma_2^K & \gamma_1^K & \gamma_0^K & \gamma_{-1}^K & \gamma_{-2}^K & 0 & & \cdots & & 0\\
				0 & & & & & & & & & \\
				\vdots & \ddots & \ddots & \ddots & \ddots & \ddots & \ddots & \ddots & & \vdots \\
				& & & & & & & & & 0 \\
				0 & & \cdots & & 0 & \gamma_2^{-K} & \gamma_1^{-K} & \gamma_0^{-K} & \gamma_{-1}^{-K} & \gamma_{-2}^{-K}
			\end{bsmallmatrix}.
		\end{align}
		Here, the coefficients of $\Gamma_n^{\alpha}(\omega)$ are given by
		\begin{align}
			&\gamma_2^n:=\frac{\varepsilon_{\kappa}\varepsilon_{\srs}}{4}\left(\omega+(n+2)\Omega\right)^2-\Omega\frac{\varepsilon_{\kappa}\varepsilon_{\srs}}{4}\left(\omega+(n+2)\Omega\right),\nonumber\\
			&\gamma_1^n:=\mathcal{C}^{\alpha}\frac{\varepsilon_{\srs}}{2}-\Omega\frac{\varepsilon_{\kappa}}{2}\left(\omega+(n+1)\Omega\right)+\frac{\varepsilon_{\kappa}+\varepsilon_{\srs}}{2}\left(\omega+(n+1)\Omega\right)^2,\nonumber\\
			&\gamma_0^n:=\mathcal{C}^{\alpha}+\left(\frac{\varepsilon_{\srs}\varepsilon_{\kappa}}{2}+1\right)\left(\omega+n\Omega\right)^2,\nonumber\\
			&\gamma_{-1}^n:=\mathcal{C}^{\alpha}\frac{\varepsilon_{\srs}}{2}+\Omega\frac{\varepsilon_{\kappa}}{2}\left(\omega+(n-1)\Omega\right)+\frac{\varepsilon_{\kappa}+\varepsilon_{\srs}}{2}\left(\omega+(n-1)\Omega\right)^2,\nonumber\\
			&\gamma_{-2}^n:=\frac{\varepsilon_{\kappa}\varepsilon_{\srs}}{4}\left(\omega+(n-2)\Omega\right)^2+\Omega\frac{\varepsilon_{\kappa}\varepsilon_{\srs}}{4}\left(\omega+(n-2)\Omega\right).
			\nonumber
		\end{align}
	\end{prop}
	\begin{rem}
		Note that in the above proposition we use the capacitance matrix for $N=1$ resonator, \textit{i.e.}, $\mathcal{C}^{\alpha}$ is a scalar given by \cite[Eq. (A2)]{jinghao_liora}:
		\begin{align}
			\mathcal{C}^{\alpha}=\frac{\delta\kappa_{\mathrm{r}}}{\ell_1\rho_{\mathrm{r}}}\left(\frac{1}{\ell_{01}}+\frac{1}{\ell_{12}}-\frac{2}{\ell_{12}}\mathrm{cos}\left(\alpha L\right)\right).
		\end{align}
	\end{rem}
	\begin{proof}
		Applying the inverse Floquet-Bloch transform to \eqref{eq:pre_capreal} allows us to obtain
		\begin{align}
			\mathcal{I}^{-1}\left[(I+H^m)\mathcal{C}^{\alpha}\left(\sum\limits_{n\in\mathbb{Z}}\mathbf{v}_n^{\alpha}\mathrm{e}^{\mathrm{i}(\omega+n\Omega)t}\frac{1}{\srs(t)}\right)-\frac{1}{\srs(t)}\frac{\mathrm{d}}{\mathrm{d}t}\left(\frac{1}{\kappa(t)}\frac{\mathrm{d}}{\mathrm{d}t}\sum\limits_{n\in\mathbb{Z}}\mathbf{v}_n^{\alpha}\mathrm{e}^{\mathrm{i}(\omega+n\Omega)t}\right)\right]=0,
		\end{align}
		where we used that $B^m=I+H^m$. By linearity of $\mathcal{I}^{-1}$ and upon evaluating the derivatives, we arrive at
		\begin{align}
			\sum\limits_{n\in\mathbb{Z}}\mathrm{e}^{\mathrm{i}(\omega+n\Omega)t}\mathcal{I}^{-1}\left[(I+H^m)\mathcal{C}^{\alpha}\mathbf{v}^{\alpha}_n\frac{1}{\srs(t)}-\mathrm{i}\frac{1}{\srs(t)}\frac{\mathrm{d}}{\mathrm{d}t}\frac{1}{\kappa(t)}\mathbf{v}_n^{\alpha}(\omega+n\Omega)+\frac{1}{\srs(t)\kappa(t)}\mathbf{v}_n^{\alpha}(\omega+n\Omega)^2\right]=0.
		\end{align}
		Next, we substitute the specific definition of $\kappa(t)$ and $\srs(t)$ \eqref{def:rho_kappa} into the equation above and expand $1/\srs(t)$ and $1/\kappa(t)$:
		\begin{align}
			\sum\limits_{n\in\mathbb{Z}}\mathrm{e}^{\mathrm{i}(\omega+n\Omega)t}\mathcal{I}^{-1}\Bigl[&(I+H^m)\mathcal{C}^{\alpha}\mathbf{v}_n^{\alpha}\left(\frac{\varepsilon_{\srs}}{2}\mathrm{e}^{\mathrm{i}\Omega t}+1+\frac{\varepsilon_{\srs}}{2}\mathrm{e}^{-\mathrm{i}\Omega t}\right)\nonumber\\
			&-\mathrm{i}\left(\frac{\varepsilon_{\srs}}{2}\mathrm{e}^{\mathrm{i}\Omega t}+1+\frac{\varepsilon_{\srs}}{2}\mathrm{e}^{-\mathrm{i}\Omega t}\right)\left(\frac{\mathrm{i}\Omega\varepsilon_{\kappa}}{2}\mathrm{e}^{\mathrm{i}\Omega t}-\frac{\mathrm{i}\Omega\varepsilon_{\kappa}}{2}\mathrm{e}^{-\mathrm{i}\Omega t}\right)\mathbf{v}_n^{\alpha}(\omega+n\Omega)\nonumber\\
			&+\left(\frac{\varepsilon_{\srs}}{2}\mathrm{e}^{\mathrm{i}\Omega t}+1+\frac{\varepsilon_{\srs}}{2}\mathrm{e}^{-\mathrm{i}\Omega t}\right)\left(\frac{\varepsilon_{\kappa}}{2}\mathrm{e}^{\mathrm{i}\Omega t}+1+\frac{\varepsilon_{\kappa}}{2}\mathrm{e}^{-\mathrm{i}\Omega t}\right)\mathbf{v}_n^{\alpha}\left(\omega+n\Omega\right)^2\Bigr]=0.
		\end{align}
		Rearranging the sum appropriately leads to 
		\begin{align}
			\sum\limits_{n\in\mathbb{Z}}\mathrm{e}^{\mathrm{i}(\omega+n\Omega)t}\mathcal{I}^{-1}\Bigl[&\left(\frac{\varepsilon_{\kappa}\varepsilon_{\srs}}{4}(\omega+(n-2)\Omega)^2+\Omega\frac{\varepsilon_{\kappa}\varepsilon_{\srs}}{4}(\omega+(n-2)\Omega)\right)\mathbf{v}_{n-2}^{\alpha}\nonumber\\
			&+\left(\mathcal{C}^{\alpha}\frac{\varepsilon_{\srs}}{2}+\Omega\frac{\varepsilon_{\kappa}}{2}(\omega+(n-1)\Omega)+\frac{\varepsilon_{\kappa}+\varepsilon_{\srs}}{2}(\omega+(n-1)\Omega)^2\right)\mathbf{v}_{n-1}^{\alpha}\nonumber\\
			&+\left(\mathcal{C}^{\alpha}+\left(\frac{\varepsilon_{\kappa}\varepsilon_{\srs}}{2}+1\right)(\omega+n\Omega)^2\right)\mathbf{v}_n^{\alpha}\nonumber\\
			&+\left(\mathcal{C}^{\alpha}\frac{\varepsilon_{\srs}}{2}-\Omega\frac{\varepsilon_{\kappa}}{2}(\omega+(n+1)\Omega)+\frac{\varepsilon_{\kappa}+\varepsilon_{\srs}}{2}(\omega+(n+1)\Omega)^2\right)\mathbf{v}_{n+1}^{\alpha}\nonumber\\
			&+\left(\frac{\varepsilon_{\kappa}\varepsilon_{\srs}}{4}(\omega+(n+2)\Omega)^2-\Omega\frac{\varepsilon_{\kappa}\varepsilon_{\srs}}{4}(\omega+(n+2)\Omega)\right)\mathbf{v}_{n+2}^{\alpha}\nonumber\\
			&+H^m\mathcal{C}^{\alpha}\left(\frac{\varepsilon_{\srs}}{2}\mathbf{v}_{n-1}^{\alpha}+\mathbf{v}_n^{\alpha}+\frac{\varepsilon_{\srs}}{2}\mathbf{v}_{n+1}^{\alpha}\right)\Bigr]=0.
		\end{align}
		This is only true if every term in the above sum is zero for every $n\in\mathbb{Z}$. Next, we introduce the matrices $\Gamma^{\alpha}_n(\omega):=\begin{bsmallmatrix}
			\gamma_2^K & \gamma_1^K & \gamma_0^K & \gamma_{-1}^K & \gamma_{-2}^K
		\end{bsmallmatrix}$ and $\mathcal{G}^m$, which are defined such that
		\begin{align}\label{eq:IGammaCG}
			\mathcal{I}^{-1}\Bigl[\Gamma^{\alpha}_n(\omega)\mathbf{w}_n^{\alpha}+\mathcal{C}^{\alpha}\mathcal{G}^m\mathbf{w}_n^{\alpha}\Bigr]=0,\quad\forall\,n\in\mathbb{Z},
		\end{align}
		where
		\begin{align}
			\mathbf{w}_n^{\alpha}:=\begin{bsmallmatrix}\\
				\mathbf{v}^{\alpha}_{n+2} \\ \mathbf{v}^{\alpha}_{n+1} \\ \mathbf{v}^{\alpha}_{n} \\ \mathbf{v}^{\alpha}_{n-1} \\ \mathbf{v}^{\alpha}_{n-2}\\
			\end{bsmallmatrix}.
		\end{align}
		With a truncation parameter $K\in\mathbb{Z}$, we reformulate \eqref{eq:IGammaCG} as one equation independent of $n\in\mathbb{Z}$:
		\begin{align}
			\mathcal{I}^{-1}\Bigl[\Gamma^{\alpha}(\omega)\mathbf{w}^{\alpha}+\mathcal{C}^{\alpha}G^m\mathbf{w}^{\alpha}\Bigr]=0, \quad \mathbf{w}^{\alpha}:=\begin{bsmallmatrix}
				\\ \mathbf{v}^{\alpha}_{K} \\ \vdots \\ \mathbf{v}^{\alpha}_{-K} \\
			\end{bsmallmatrix}.
		\end{align}
		Using linearity of $\mathcal{I}^{-1}$ and applying the inverse Floquet-Bloch transform to \eqref{eq:IGammaCG} allows us to write
		\begin{align}
			\Gamma^{\alpha}\mathbf{w}^{\alpha}=\sum\limits_{m=0}^M\mathbf{c}_m\mathrm{e}^{\mathrm{i}\alpha m},
		\end{align}
		where the Fourier coefficients $\mathbf{c}_m$ are given by
		\begin{align}
			\mathbf{c}_m&=-\mathcal{I}^{-1}\bigl[\mathcal{C}^{\alpha}G^m\mathbf{w}^{\alpha}\bigl]\nonumber\\
			&=-\frac{L}{2\pi}\int_{-\pi/L}^{\pi/L}\mathcal{C}^{\alpha}G^m\mathbf{w}^{\alpha}\mathrm{e}^{-\mathrm{i}\alpha m}\,\mathrm{d}\alpha\nonumber\\
			&=-\mathcal{C}^{\alpha}G^m\sum\limits_{k=0}^M\frac{L}{2\pi}\int_{-\pi/L}^{\pi/L}\left(\Gamma^{\alpha}\right)^{-1}\mathbf{c}_k\mathrm{e}^{\mathrm{i}\alpha(k-m)}\,\mathrm{d}\alpha.
		\end{align}
		From the above characterisation of $\mathbf{c}_m$, \eqref{eq:det_cond} directly follows.
	\end{proof}
	
	\subsection{Time localisation}
	To achieve localisation in time, we assume $\kappa$ to have a defect in its temporal periodicity, namely for
	\begin{align}
		\kappa(x,t)=\begin{cases}
			\kappa_0,&x\notin D;\\
			\kappa_{\mathrm{r}}\kappa_i(t),&x\in D^m_i, 
		\end{cases}
	\end{align}
	we assume 
	\begin{align}\label{eq:kappa_nonper}
		\frac{1}{\kappa_i(t)}=\frac{1}{\kappa_i^p(t)}+c_if(t)=\sum\limits_{m=-M}^Mk_{i,m}\mathrm{e}^{\mathrm{i}m\Omega t}+c_if(t),
	\end{align}
	for a function $f(t)$ which decays sufficiently fast in time.\par 
	Under the above assumption, the total wave field $u(x,t)$ can be expressed through the inverse Fourier transform as follows:
	\begin{align}
		u(x,t)=\mathcal{F}^{-1}\left[v(x,\omega)\right]=\int_{-\infty}^{\infty}v(x,\omega)\mathrm{e}^{\mathrm{i}\omega t}\,\mathrm{d}\omega.
	\end{align}
	Analogously to \cite[Lemma 4.1]{erik_JCP}, we prove the following result:
	\begin{lemma}\label{lemma:v_const}
		As $\delta\to0$, the functions $v(x,\omega)$ are approximately constant in $x$ inside each resonator $D_i^m$:
		\begin{align}
			v(x,\omega)=v^m_i(\omega)+O(\delta^{(1-\gamma)/2}),\quad x\in D_i^m,
		\end{align}
		for some $0<\gamma<1$. We write $u^m_i(t)=\mathcal{F}^{-1}\left[v^m_i(\omega)\right]$.
	\end{lemma}
	Based on this, we can prove the following proposition:
	\begin{prop}\label{prop:cap_vmi_st}
		Any localised solution $u$ to \eqref{eq:dD_system_u} corresponding to a subwavelength quasifrequency $\omega$ satisfies the following ODE:
		\begin{align}
			\sum\limits_{j=1}^N\frac{1}{\srs_j(t)}C_{ij}^{\alpha}u_j(t)=\frac{\rho_{\mathrm{r}}\ell_i}{\delta\kappa_{\mathrm{r}}\srs_i(t)}\frac{\mathrm{d} }{\mathrm{d} t}\left(\left(\frac{1}{\kappa_i^p(t)}+c_if(t)\right)\frac{\mathrm{d}}{\mathrm{d} t}u_i(t)\right),
		\end{align}
		for all $i=1,\dots,N$.
	\end{prop}
	\begin{proof}
		Applying the Fourier transform $\mathcal{F}$ to \eqref{eq:dD_system_u} leads to
		\begin{align}\label{eq:dD_system_u_Fourier}
			\left\{
			\begin{array} {ll}
				\ds -\frac{\omega^2}{\left(v_0\right)^2}v(x,\omega) - \Delta v(x,\omega) = 0, \quad &x\notin D, \\[1em]
				\ds \frac{\mathrm{i}\omega}{\left(v_{\mathrm{r}}\right)^2}\left(\sum\limits_{n=-M}^M\tilde{k}_{i,n}(\omega+n\Omega)v(x,\omega+n\Omega)+\tilde{c}_i\int_{-\infty}^{\infty}\hat{f}(\tau)\mathrm{i}(\omega-\tau)v(x,\omega-\tau)\,\mathrm{d}\tau\right) \\
				\qquad\qquad -  \Delta v(x,\omega) = 0, \quad &x\in D_i^m,\\[1em]
				\ds v|_{\pm}(x_i^{\pm},\omega) = \sum\limits_{n=-M}^M\tilde{\srsf}_{i,n}v|_{\mp}(x_i^{\pm},\omega+n\Omega), \quad &i=1,\dots,N, \\[1em] 	
				\sum\limits_{n=-M}^M\tilde{\srsf}_{i,n}\ds \frac{\p v}{\p x }\bigg|_{\mp}(x_i^{\pm},\omega+n\Omega) = \delta\frac{\p v}{\p x }\bigg|_{\pm}(x_i^{\pm},\omega), &i=1,\dots,N, \\[1em]
			\end{array}		
			\right.
		\end{align}
		where $\hat{f}(\omega):=\mathcal{F}\left[f(t)\right]$. The coefficients $\tilde{k}_{i,n}$ and $\tilde{\srsf}_{i,n}$ denote the coefficients of the Fourier series of $1/\kappa_i(t)$ and $1/\srs_i(t)$, respectively, multiplied by constants emerging from applying the Fourier transform. Applying the operator $I_{\p D_i}$ to $v$, we obtain
		\begin{align}
			I_{\p D_i}[v]&=-\frac{1}{\delta}\sum\limits_{n=-M}^M\tilde{\srsf}_{i,n}\int_{x_i^-}^{x_i^+}\frac{\p^2 v(x,\omega+n\Omega)}{\p x^2}\,\mathrm{d}x\nonumber\\
			&=\frac{\rho_{\mathrm{r}}}{\kappa_{\mathrm{r}}\delta}\sum\limits_{n=-M}^M\tilde{\srsf}_{i,n}\mathrm{i}(\omega+n\Omega)\left(\sum\limits_{l=-M}^M\tilde{k}_{i,l}(\omega+(n+l)\Omega)\int_{x_i^-}^{x_i^+}v(x,\omega+(n+l)\Omega)\,\mathrm{d}x\right.\nonumber\\
			&\qquad\qquad\left.+\tilde{c}_i\int_{-\infty}^{\infty}\hat{f}(\tau)\mathrm{i}(\omega+n\Omega-\tau)\int_{x_i^-}^{x_i^+}v(x,\omega+n\Omega-\tau)\,\mathrm{d}x\,\mathrm{d}\tau\right)\nonumber\\
			&=\frac{\ell_i}{\delta\left(v_{\mathrm{r}}\right)^2}\sum\limits_{n=-M}^M\tilde{\srsf}_{i,n}\mathrm{i}(\omega+n\Omega)\left(\sum\limits_{l=-M}^M\tilde{k}_{i,l}(\omega+(n+l)\Omega)v_i(\omega+(n+l)\Omega)\right.\nonumber\\
			&\qquad\qquad\left.+\tilde{c}_i\int_{-\infty}^{\infty}\hat{f}(\tau)\mathrm{i}(\omega+n\Omega-\tau)v_i(\omega+n\Omega-\tau)\,\mathrm{d}\tau\right).
		\end{align}
		In the last step we used the result from Lemma \ref{lemma:v_const}.\par  
		On the other hand, using the definition of the capacitance matrix allows us to write
		\begin{align}
			I_{\partial D_i}[v^{\alpha}]=\sum\limits_{n=-\infty}^{\infty}\sum\limits_{j=1}^N\tilde{\srsf}_{i,n}C_{ij}^{\alpha}v_j(\omega+n\Omega).
		\end{align}
		By applying the inverse Fourier transform $\mathcal{F}^{-1}$ to both expressions of $I_{\p D_j}[v]$, we obtain the desired result using the convolution rule for the Fourier transform.
	\end{proof}
	Next, we rewrite this system of ODEs as a single vector ODE with the matrices 
	\begin{align}
		\Srs(t):=\mathrm{diag}\left(\srs_i(t)\right)_{i=1}^N,\quad K(t):=\mathrm{diag}\left(\kappa_i^p(t)\right)_{i=1}^N,\quad \widetilde{C}(t):=\mathrm{diag}\left(c_if(t)\right)_{i=1}^N
	\end{align}
	and the vector
	\begin{align}
		\mathbf{v}(\omega):=\left[v_i(\omega)\right]_{i=1}^N,\quad \mathbf{u}(t):=\int_{-\infty}^{\infty}\mathbf{v}(\omega)\mathrm{e}^{\mathrm{i}\omega t}\,\mathrm{d}\omega.
	\end{align}
	Using this notation, we obtain
	\begin{align}\label{eq:v_timedefect_ODE}
		\Srs^{-1}(t)\mathcal{C}^{\alpha}\mathbf{u}(t)=\Srs^{-1}(t)\frac{\mathrm{d}}{\mathrm{d}t}\left(\left(K^{-1}(t)+\widetilde{C}(t)\right)\frac{\mathrm{d}}{\mathrm{d}t}\mathbf{u}(t)\right).
	\end{align}
	
	\subsection{Space-time localisation}
	To achieve localisation in space and time simultaneously, we assume $\kappa$ to have a defect in time and $\rho$ a defect in space. We assume $\kappa_i(t)$ to be defined as in \eqref{eq:kappa_nonper} and $\rho$ with a defect in space as stated in \eqref{eq:rho_nonper}. Next, we generalise \cite[Proposition 4.2]{ammarihiltunen2020edge} and obtain the following result:
	\begin{prop}\label{prop:cap_vmi_st2}
		Any localised solution $u$ to \eqref{eq:dD_system_u}, corresponding to a subwavelength quasifrequency $\omega$, satisfies the following ODE:
		\begin{align}
			\sum\limits_{j=1}^N\frac{1}{\srs_j(t)}C_{ij}^{\alpha}\sum\limits_{m=-\infty}^{\infty}u^m_j(t)\mathrm{e}^{\mathrm{i}\alpha mL}=\frac{\ell_i}{\kappa_{\mathrm{r}}\srs_i(t)}\frac{\mathrm{d} }{\mathrm{d} t}\left(\left(\frac{1}{\kappa_i^p(t)}+c_if(t)\right)\frac{\mathrm{d}}{\mathrm{d} t}\sum\limits_{m=-\infty}^{\infty}u^m_i(t)\frac{\mathrm{e}^{\mathrm{i}\alpha mL}\rho_{\mathrm{r}}^{m,i}}{\delta^m}\right),
		\end{align}
		for all $i=1,\dots,N$.
	\end{prop}
	\begin{proof}
		The proof follows the same method as that of \Cref{prop:conv_cap_real} and \Cref{prop:cap_vmi_st}. We consider a system of infinitely many unit cells, each containing $N$ resonators $\left(D_i^m\right)_{i=1,\dots,N}$ for all $m\in\mathbb{Z}$ and the wave propagation to be governed by \eqref{eq:dD_system_u}. Applying the Fourier transform $\mathcal{F}$ to \eqref{eq:dD_system_u} leads to \eqref{eq:dD_system_u_Fourier}.\par 
		Next, we apply the Floquet-Bloch transform to \eqref{eq:dD_system_u_Fourier} and obtain
		\begin{align}\label{eq:dD_system_u_Fourier_FloquetBloch}
			\left\{
			\begin{array} {ll}
				\ds -\frac{\omega^2}{\left(v_0\right)^2}v^{\alpha}(x,\omega) - \Delta v^{\alpha}(x,\omega) = 0, & x\notin D, \\[1em]
				\ds \frac{\mathrm{i}\omega}{\kappa_{\mathrm{r}}}\sum\limits_{m=-\infty}^{\infty}\rho_{\mathrm{r}}^{m,i}\mathrm{e}^{\mathrm{i}\alpha mL}\Bigg(\sum\limits_{n=-M}^M\tilde{k}_{i,n}(\omega+n\Omega)v^{\alpha}(x,\omega+n\Omega)\\
				\ds \qquad\qquad +\tilde{c}_i\int_{-\infty}^{\infty}\hat{f}(\tau)\mathrm{i}(\omega-\tau)v^{\alpha}(x,\omega-\tau)\,\mathrm{d}\tau\Bigg) -  \Delta v^{\alpha}(x,\omega) = 0, & x\in D_i,\\[1em]
				\ds v^{\alpha}|_{\pm}(x_i^{\pm},\omega) = \sum\limits_{n=-M}^M\tilde{\srsf}_{i,n}v^{\alpha}|_{\mp}(x_i^{\pm},\omega+n\Omega), & i=1,\dots,N, \\[1em] 	
				\sum\limits_{n=-M}^M\tilde{\srsf}_{i,n}\ds \frac{\p v^{\alpha}}{\p x }\bigg|_{\mp}(x_i^{\pm},\omega+n\Omega) = \sum\limits_{m=-\infty}^{\infty}\delta^m\mathrm{e}^{\mathrm{i}\alpha mL}\frac{\p v^{\alpha}}{\p x }\bigg|_{\pm}(x_i^{\pm},\omega), &i=1,\dots,N. \\[1em]
			\end{array}		
			\right.
		\end{align}
		Next, we apply the operator $I_{\partial D_i}$ to $v^{\alpha}$:
		\begin{align}
			I_{\partial D_i}[v^{\alpha}]&=-\frac{1}{\delta}\sum\limits_{n=-M}^M\tilde{\srsf}_{i,n}\int_{x_i^-}^{x_i^+}\frac{\p^2 v^{\alpha}(x,\omega+n\Omega)}{\p x^2}\,\mathrm{d}x\nonumber\\
			&=\frac{1}{\kappa_{\mathrm{r}}}\sum\limits_{n=-M}^M\tilde{\srsf}_{i,n}\mathrm{i}(\omega+n\Omega)\left(\sum\limits_{l=-M}^M\tilde{k}_{i,l}(\omega+(n+l)\Omega)\int_{x_i^-}^{x_i^+}v^{\alpha}(x,\omega+(n+l)\Omega)\,\mathrm{d}x\right.\nonumber\\
			&\left.\qquad\qquad+\tilde{c}_i\int_{-\infty}^{\infty}\hat{f}(\tau)\mathrm{i}(\omega+n\Omega-\tau)\int_{x_i^-}^{x_i^+}v^{\alpha}(x,\omega+n\Omega-\tau)\,\mathrm{d}x\,\mathrm{d}\tau\right)\sum\limits_{m=-\infty}^{\infty}\frac{\mathrm{e}^{\mathrm{i}\alpha mL}\rho_{\mathrm{r}}^{m,i}}{\delta^m}\nonumber\\
			&=\frac{\ell_i}{\kappa_{\mathrm{r}}}\sum\limits_{n=-M}^M\tilde{\srsf}_{i,n}\mathrm{i}(\omega+n\Omega)\left(\sum\limits_{l=-M}^M\tilde{k}_{i,l}(\omega+(n+l)\Omega)\sum\limits_{m=-\infty}^{\infty}\frac{v^{m}_i(\omega+(n+l)\Omega)\mathrm{e}^{\mathrm{i}\alpha mL}\rho_{\mathrm{r}}^{m,i}}{\delta^m}\right.\nonumber\\
			&\left.\qquad\qquad+\tilde{c}_i\int_{-\infty}^{\infty}\hat{f}(\tau)\mathrm{i}(\omega+n\Omega-\tau)\sum\limits_{m=-\infty}^{\infty}\frac{v_i^{m}(\omega+n\Omega-\tau)\mathrm{e}^{\mathrm{i}\alpha mL}\rho_{\mathrm{r}}^{m,i}}{\delta^m}\,\mathrm{d}\tau\right).
		\end{align}
		In the last step we used the result from Lemma \ref{lemma:v_const}.\par  
		On the other hand, using the definition of the capacitance matrix allows us to write
		\begin{align}
			I_{\partial D_i}[v^{\alpha}]=\sum\limits_{n=-\infty}^{\infty}\sum\limits_{j=1}^N\tilde{\srsf}_{i,n}C_{ij}^{\alpha}\sum\limits_{n=-\infty}^{\infty}v_j^m(\omega+l\Omega)\mathrm{e}^{\mathrm{i}\alpha mL}.
		\end{align}
		By applying the inverse Fourier transform $\mathcal{F}^{-1}$ to both expressions of $I_{\partial D_j}[v^{\alpha}]$, we obtain the desired result using the convolution rule for the Fourier transform.
	\end{proof}
	We use the same notation of $b^m_i$ as defined in \eqref{eq:bmi} with the matrices $B^m,\,\Srs(t),\,K(t),\,\widetilde{C}(t)$ 
	and the vector $\mathbf{v}^m(\omega):=\left[v^m_i(\omega)\right]_{i=1}^N$ as before.
	We shall now use the ODE derived in Proposition \ref{prop:cap_vmi_st2} and the multiplication-rule proved in Lemma \ref{lemma:multip} to provide a capacitance matrix approximation in the real space, rather than the dual space, as in the static case in \cite[Proposition 3.5]{ammari2023anderson}.
	\begin{prop}\label{prop:conv_cap_real2}
		Any localised solution to \eqref{eq:dD_system_u}, corresponding to a subwavelength frequency $\omega$ satisfies
		\begin{align}\label{eq:conv_cap_real_st}
			B^m\sum\limits_{n\in\Lambda}\mathcal{C}^{m-n}\int_{-\infty}^{\infty} 
			\Srs^{-1}(t) \mathbf{v}^n(\omega)\mathrm{e}^{\mathrm{i}\omega t}\,\mathrm{d}\omega=\Srs^{-1}(t)\frac{\mathrm{d}}{\mathrm{d}t}\left(\left( K^{-1}(t)+ \widetilde{C}(t)\right)\frac{\mathrm{d}}{\mathrm{d}t}\left(\int_{-\infty}^{\infty}\mathbf{v}^m(\omega)\mathrm{e}^{\mathrm{i}\omega t}\,\mathrm{d}\omega\right)\right),
		\end{align}
		for every $m\in\Lambda$, where $\mathcal{C}^m=\mathcal{I}^{-1}[\mathcal{C}^{\alpha}](m)$.
	\end{prop}
	\begin{proof}
		If we assume the material to be perturbed in space, the formulation obtained in Proposition \ref{prop:cap_vmi_st} can be rewritten as
		\begin{align}\label{eq:pre_capreal_st}
			&\mathcal{C}^{\alpha}\sum\limits_{m\in\mathbb{Z}}\int_{-\infty}^{\infty} \Srs^{-1}(t) \mathbf{v}^m(\omega) \mathrm{e}^{\mathrm{i}\omega t}\,\mathrm{d}\omega\,\mathrm{e}^{\mathrm{i}\alpha mL}\nonumber\\
			&\qquad=\Srs^{-1}(t)\frac{\mathrm{d}}{\mathrm{d}t}\left(\left(K^{-1}(t)+ \widetilde{C}(t)\right)\frac{\mathrm{d}}{\mathrm{d}t}\sum\limits_{m\in\mathbb{Z}}\int_{-\infty}^{\infty}
			\left(B^m\right)^{-1} \mathbf{v}^{m}(\omega)\mathrm{e}^{\mathrm{i}\omega t}\,\mathrm{d}\omega\,\mathrm{e}^{\mathrm{i}\alpha mL}\right),
		\end{align}
		where $\mathcal{C}^{\alpha}$ is the generalised capacitance matrix. Lastly, apply the inverse Floquet-Bloch transform and use Lemma \ref{lemma:multip} to obtain the desired result.
	\end{proof}
	Equivalently, \eqref{eq:conv_cap_real_st} can be expressed more concisely as a vector ODE with the previously introduced matrices $\mathcal{C},\,B,\,\Srsc(t),\,\mathcal{K}(t)$, the infinitely diagonal matrix $\widetilde{C}(t)$ and the vector
	\begin{align}
		\mathbf{u}(t):=\begin{bsmallmatrix}
			\vdots\\\mathbf{u}^{-1}(t)\\\mathbf{u}^{0}(t)\\\mathbf{u}^{1}(t)\\\vdots \\
		\end{bsmallmatrix}, \quad \mathbf{u}^m(t):=\int_{-\infty}^{\infty}\mathbf{v}^m(\omega)\mathrm{e}^{\mathrm{i}\omega t}\,\mathrm{d}\omega.
	\end{align}
	As in \eqref{eq:v_ODE}, we may write
	\begin{align}\label{eq:v_ODE_st}
		B\mathcal{C}^{\alpha}\Srsc^{-1}(t)\mathbf{u}(t)=\Srsc^{-1}(t)\frac{\mathrm{d}}{\mathrm{d}t}\left(\left(\mathcal{K}^{-1}(t)+\widetilde{C}(t)\right)\frac{\mathrm{d}}{\mathrm{d}t}\mathbf{u}(t)\right),\qquad \text{for}\,\,t\in\mathbb{R}.
	\end{align}
	This system can be rewritten into a system of first-order ODEs, which leads to the definition of a $2N\times2N$ - matrix $\mathcal{M}^{\alpha}(t)$, whose eigenvalues are denoted by $\lambda_i^{\alpha}$. Moreover, note that for the numerical implementation of this ODE one must truncate the system with a truncation parameter $K\in\mathbb{N}$.
	
	\section{Numerical experiments}\label{sec:numerical_exp}

	For the numerical computation of the eigenvectors $\mathbf{v}^{\alpha}$ we cannot assume the system to be infinitely periodic. Thus, we approximate it by a super-cell with many reoccurring cells, each containing $N$ resonators.\par 
	In the static case, we determine localisation through the eigenvectors of \eqref{eq:v_ODE_st} and plotting the square of the absolute value of their entries. With the introduction of time-dependent material parameters $\srs$ and $\kappa$, the capacitance matrix approximation formula is no longer a simple eigenvalue problem, but a system of ODEs, as given in \eqref{eq:capmat_ode}. To recognise the appearance of localised waves in time-modulated metamaterials, we must therefore adapt our definition of eigenmodes. We denote the eigenvectors corresponding to the first $N$ subwavelength resonant frequencies $\left(\omega_i^{\alpha}\right)_{i=1}^N$ by $\left(\mathbf{v}_i^{\alpha}\right)_{i=1}^N\subset\mathbb{C}^{N}$ and we shall use them to determine localised modes. We quantify the degree of localisation in the case of a defect in space as follows:
	\begin{definition}[Spatial Degree of Localisation]
		We define the spatial degree of localisation upon the introduction of a defect in space by
		\begin{align}
			d_i:=\frac{||\mathbf{v}^{\alpha}_i||_{\infty}}{||\mathbf{v}^{\alpha}_i||_{2}}.
		\end{align}
		We say that an eigenmode is localised if its corresponding $d_i$ is close to $1$ and we say it is delocalised if $d_i$ is close to $0$.
	\end{definition}
	Figure \ref{fig:deg_of_loc_space} displays the spatial degree of localisation for time-dependent resonators with a defect in space at the centre of the resonator array. The numerical results show that the time-modulation causes the frequency axis to fold, which means that the spatially localised frequencies can appear anywhere inside the spectrum. Moreover, the folded frequencies and bulk frequencies hybridise, producing multiple spatially localised modes upon the introduction of time-dependent material parameters. In the case of a static material we see in Figure \ref{fig:deg_of_loc_space} that there is a single localised mode which is decoupled from the bulk.  Interestingly, when we add time-modulation, rather than coupling to the bulk and becoming delocalised, this mode remains localised and causes additional bulk modes close to this point to be localised. 
	\begin{figure}[H]
		\centering
		\includegraphics[width=0.8\textwidth]{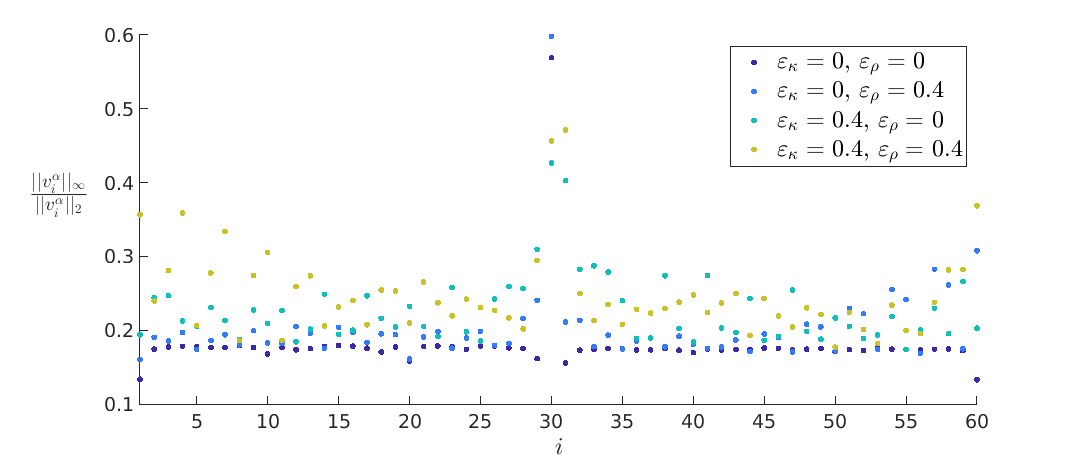}
		\caption{The spatial degree of localisation $d_i$ corresponding to $20$ reoccurring unit cells of $N=3$ resonators, where there is a defect in the wave speed inside $D^0_2$. Each resonator is of length $\ell=1$ and the spacing between two neighbouring resonators is $\ell_{ij}=2$.}
		\label{fig:deg_of_loc_space}
	\end{figure}
	\par
	Equivalently to the definition of $d_i$, we now introduce a measure for the temporal localisation. To influence the time instance at which the space localisation takes place, we implement the defect in time as follows:
	\begin{align}
		\frac{1}{\kappa_i(t)}=\sum\limits_{m=-M}^Mk_{i,m}\mathrm{e}^{\mathrm{i}m\Omega t}+c_i\mathrm{e}^{-\left(\frac{t}{t_0}-1\right)^2},\qquad \text{for}\,\,t_0\in[0,T],
	\end{align}
	\textit{i.e.} for fixed $t_0$, we chose $f(t):=\mathrm{e}^{-\left(t/t_0-1\right)^2}$.
	\begin{definition}[Localisation in Time]
		We say a mode is localised in time if it is an $L^2$-solution to the ODE \eqref{eq:v_ODE_st}.
	\end{definition}
	\begin{rem}
		Let $\left(\lambda^{\alpha}_i\right)$ be the absolute values of the eigenvalues of the matrix $\mathcal{M}^{\alpha}(t)$ emerging from the ODE \eqref{eq:v_ODE_st}. For a mode to be an $L^2$-solution it is necessary to look at the values $(\lambda_i^{\alpha})_{i=1}^N$ and check that they are less than $1$.
	\end{rem}
	In Figure \ref{fig:temp_deg_of_loc_space} we present the eigenvalues $\left(\lambda^{\alpha}_i\right)$ of the matrix $\mathcal{M}^{\alpha}(t)$. The modes $i$ associated with $\lambda_i^{\alpha}$ close to zero indicate a decaying mode in time, and the possibility of time-localisation.\par 
	Finally, we obtain space-time localised modes by considering the modes $i^*$ for which $\lambda_{i^*}^{\alpha}$ is close to zero and $d_{i^*}$ is close to one. Let $\mathbf{u}^*(t)$ be the space-time localised mode which solves \eqref{eq:v_ODE_st}. We introduce $d_{*}$ as a function of time:
	\begin{align}
		d_{*}(t):=\frac{||\mathbf{u}^*(t)||_{\infty}}{||\mathbf{u}^*(t)||_2},
	\end{align}
	which we call the \textit{time-dependent degree of localisation}. Note here that the norms are taken over the space variable. Indeed, recall that $\mathbf{u}^*(t)$ is a vector-valued function of time, whose entries each represent the evaluation inside the resonators $D_i^m$ of each cell, thus,  $||\cdot||_{\infty}$ and $||\cdot||_2$ denote the $\ell^{\infty}$- and $\ell^2$-norm, respectively, defined over sequences. Hence, $d_*(t)$ shows the degree of localisation at each time $t\in\mathbb{R}$ and we say a mode is localised in space and time if its corresponding $d_*(t)$ has a distinct peak, such as in Figure \ref{fig:dt_loc}. Having a peak in $d_*(t)$ means that we have a transition from spatial delocalisation into spatial localisation.\par 
    \begin{figure}[H]
		\centering
		\includegraphics[width=0.9\textwidth]{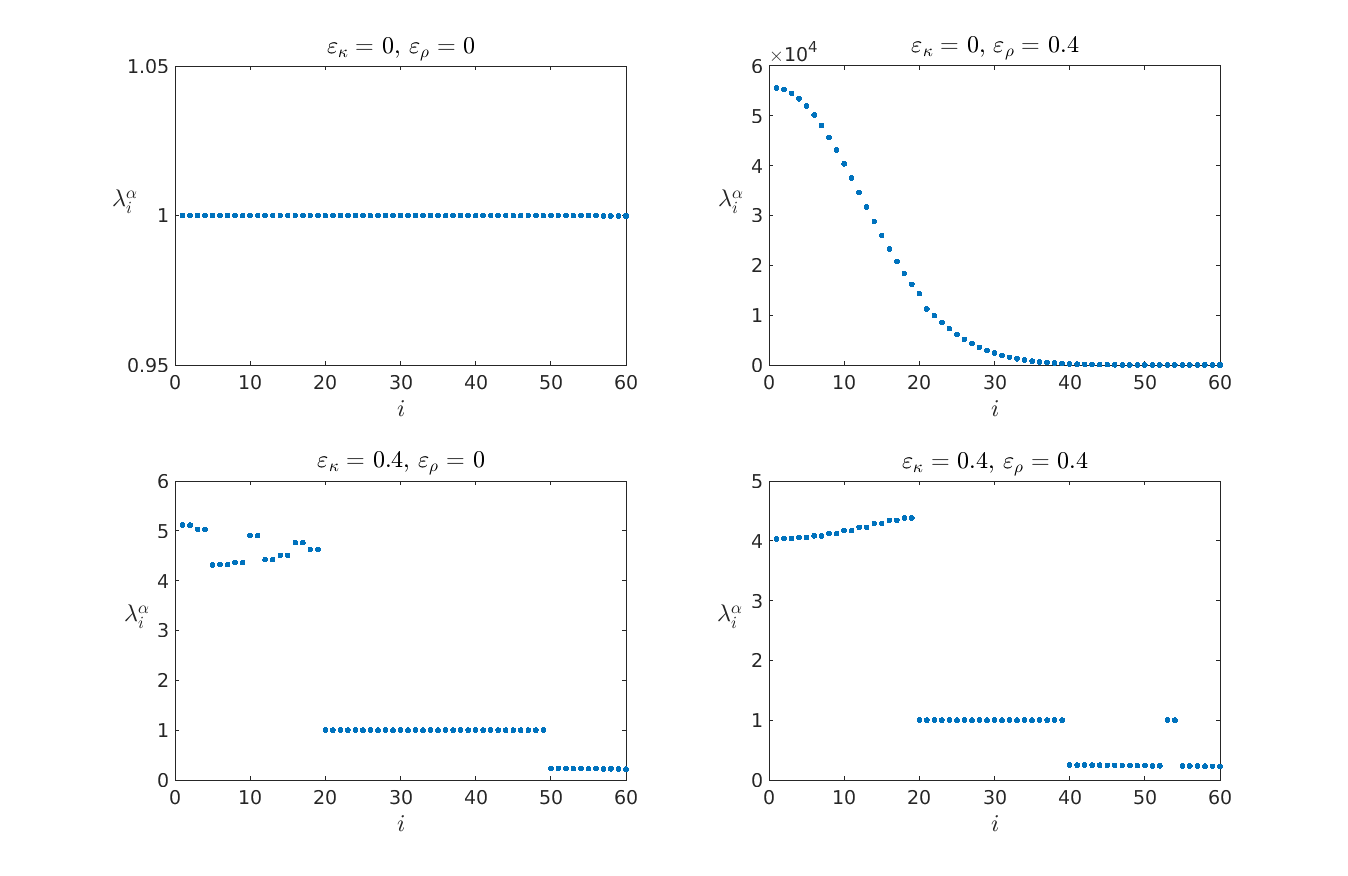}
		\caption{The eigenvalues $\left(\lambda_i^{\alpha}\right)_{i=1}^N$ corresponding to $20$ reoccurring unit cells of $N=3$ resonators, where there is a defect in the periodicity of $\kappa_2(t)$ inside $D^0_2$. Each resonator is of length $\ell=1$ and the spacing between two neighbouring resonators is $\ell_{ij}=2$. Note that in the static case $\lambda_i^{\alpha}\equiv1$, hence, it is impossible to achieve time localisation.}
		\label{fig:temp_deg_of_loc_space}
	\end{figure}
	The results presented in Figure \ref{fig:dt_loc} and Figure \ref{fig:st_loc_snapshots} present a space-time localised wave. Figure \ref{fig:dt_loc} shows that we should expect a spatially localised wave around $t=46.8\,\mathrm{s}$. The dynamics of this solution is shown in Figure \ref{fig:st_loc_snapshots}, and confirming this transition from delocalisation to localisation.
	In summary, Figure \ref{fig:dt_loc} and Figure \ref{fig:st_loc_snapshots} numerically show the existence of space-time localised waves through the introduction of a defect in space and time.\par 
	\begin{figure}[H]
		\centering
		\includegraphics[width=1.0\textwidth]{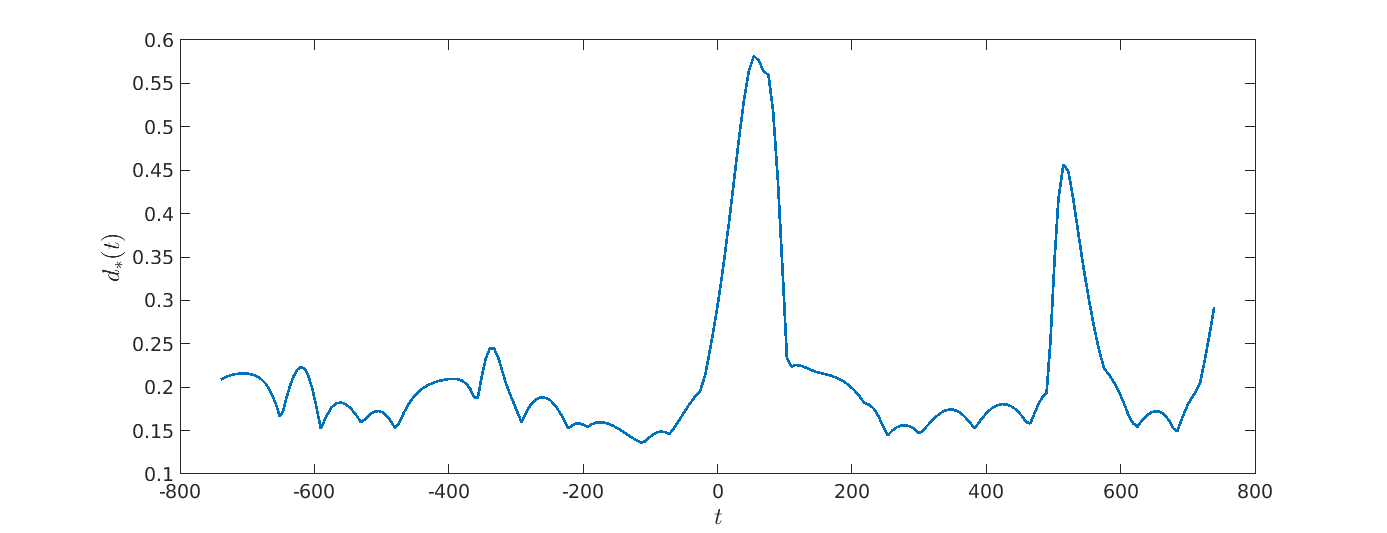}
		\caption{The time-dependent degree of localisation $d_*(t)$ corresponding to $25$ recurrences of a unit cell with $N=3$ resonators of length $\ell=1$ and spacing $\ell_{12}=\ell_{34}=1,\,\ell_{23}=2$. The material parameters are chosen to be $v_0=v_{\mathrm{r}}=1,\,\Omega=0.034,\,\delta=0.0001,\,\varepsilon_{\kappa}=0.4,\,\varepsilon_{\srs}=0.2,\,\alpha=0.01$. The material has a spatial defect at $D_2^0$, where we set $v_{\mathrm{r}}=2$, and a temporal defect at $D_2^m$ of $c_2=1$ with $t_0=T=184.7996\,\mathrm{s}$. }
		\label{fig:dt_loc}
	\end{figure}
	Note that the second peak in Figure \ref{fig:dt_loc} is due to the super-cell approximation for solving the infinite structure.
	\begin{figure}[H]
		\centering
		\includegraphics[width=1.0\textwidth]{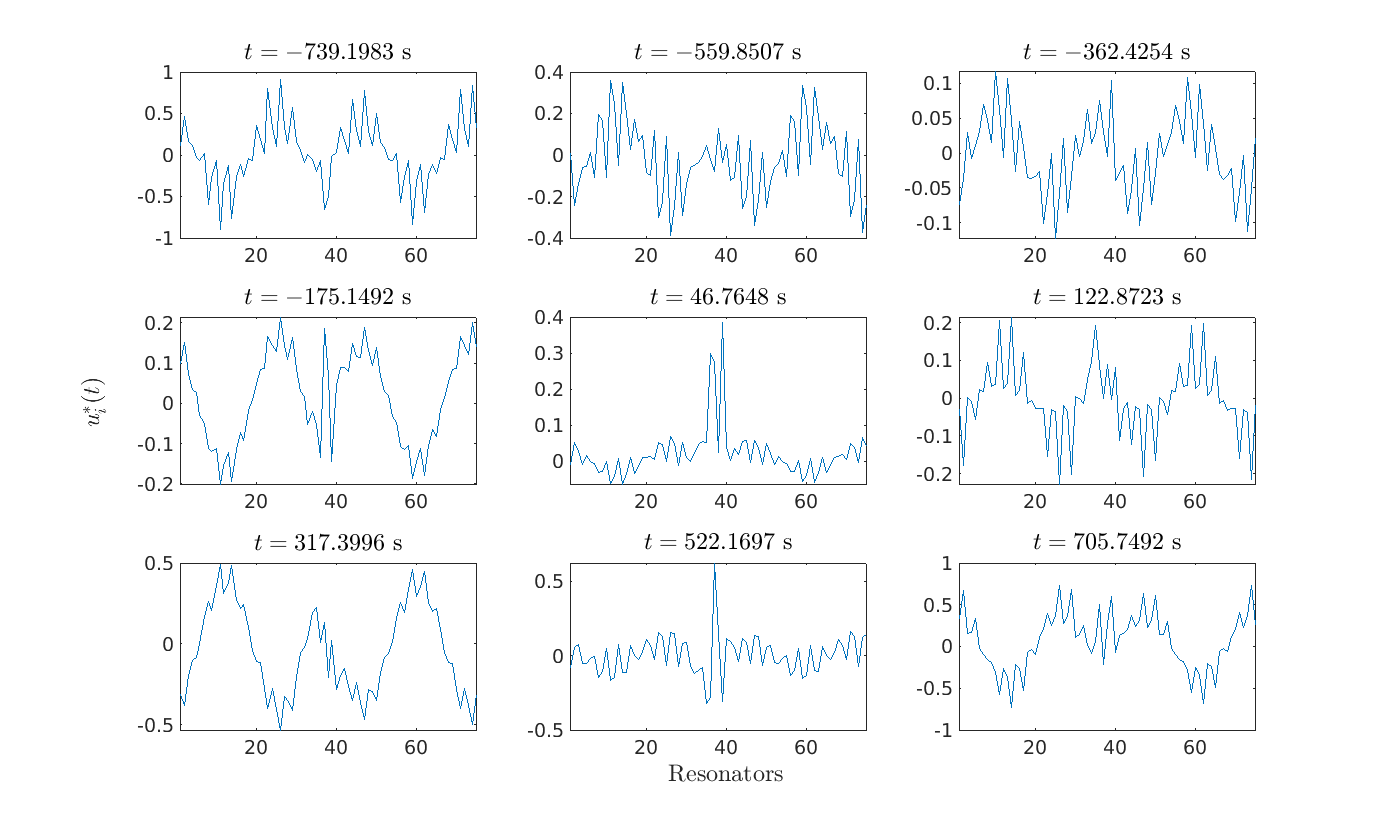}
		\caption{The localised mode corresponding to the same structure as in Figure \ref{fig:dt_loc}. Each subplot shows a snapshot of the mode over time. The fifth subplot at $t=46.7648\,\mathrm{s}$ corresponds to the peak seen in Figure \ref{fig:dt_loc}.}
		\label{fig:st_loc_snapshots}
	\end{figure}
	
	\section{Concluding remarks}\label{sec:concl}
	In this paper, we have proved the existence of space-time localised waves by introducing a defect in space in $\rho$ and a defect in time in $\kappa$. We have first considered waves which are localised in space and derived an approximation formula to the eigenmodes in the form of an ODE including the capacitance matrix $\mathcal{C}$ in \eqref{eq:v_ODE}. Moreover, we have followed the approach in \cite{ammari2023anderson} to formulate a root-finding problem, which requires the computation of the determinant of a Toeplitz matrix formulation in Proposition \ref{prop:toeplitz}.\par 
	Next, we have considered materials with a temporal inhomogeneity by breaking the periodicity of $\kappa_i$ as in \eqref{eq:kappa_nonper}. This has led to the need of rephrasing the wave $u(x,t)$ through the inverse Fourier transform instead of a Fourier series. Analogously to the spatially defected setting, we have derived an ODE approximation formula \eqref{eq:v_timedefect_ODE} to the eigenmodes.\par 
	Lastly, we have assumed the material to have a defect in space and time simultaneously, which gives rise to the ODE approximation formula \eqref{eq:v_ODE_st}. We have solved this ODE numerically, which visualises the existence of space-time localised waves in one dimensional time-varying metamaterials.\par
	Our formulations and  results in this paper can be generalised to systems of subwavelength resonators where sources of energy gain and loss represented by time-modulated complex material parameters are introduced inside the resonators \cite{fink}. They can also be used to provide the mathematical foundations of the non-Hermitian skin effect in systems of subwavelength resonators \cite{silvio_arma} achieved by breaking reciprocity with time-modulation \cite{skin_modulated}. These two very interesting directions of research will be the subjects of forthcoming works.

	\section*{Code availability}
	The codes that were used to generate the results presented in this paper are available under \url{https://github.com/rueffl/localisation_codes}.

	\addcontentsline{toc}{chapter}{References}
	\renewcommand{\bibname}{References}
	\bibliography{refs}

\begin{thebibliography}{10}

\bibitem{silvio_arma}
Habib Ammari, Silvio Barandun, Jinghao Cao, Bryn Davies, and Erik~Orvehed Hiltunen.
\newblock Mathematical foundations of the non-{H}ermitian skin effect.
\newblock {\em Arch. Ration. Mech. Anal.}, 248(3):Paper No. 33, 34, 2024.

\bibitem{jinghao-silvio2023}
Habib Ammari, Silvio Barandun, Jinghao Cao, and Florian Feppon.
\newblock Edge modes in subwavelength resonators in one dimension.
\newblock {\em Multiscale Model. Simul.}, 21(3):964--992, 2023.

\bibitem{jinghao_liora}
Habib Ammari, Jinghao Cao, Erik~Orvehed Hiltunen, and Liora Rueff.
\newblock Transmission properties of time-dependent one-dimensional metamaterials.
\newblock {\em J. Math. Phys.}, 64(12):121502, 12 2023.

\bibitem{ammari2024scattering}
Habib Ammari, Jinghao Cao, Erik~Orvehed Hiltunen, and Liora Rueff.
\newblock Scattering from time-modulated subwavelength resonators.
\newblock {\em Proc. A.}, 480(2289):Paper No. 20240177, 22, 2024.

\bibitem{jinghao2}
Habib Ammari, Jinghao Cao, and Xinmeng Zeng.
\newblock Transmission properties of space-time modulated metamaterials.
\newblock {\em Stud. Appl. Math.}, 150:558--581, 2023.

\bibitem{ammari_convrates}
Habib Ammari, Bryn Davies, and Erik~Orvehed Hiltunen.
\newblock Convergence rates for defect modes in large finite resonator arrays.
\newblock {\em SIAM J. Math. Anal.}, 55(6):7616--7634, 2023.

\bibitem{ammari2023spectralconvergencelargefinite}
Habib Ammari, Bryn Davies, and Erik~Orvehed Hiltunen.
\newblock Spectral convergence in large finite resonator arrays: the essential spectrum and band structure.
\newblock {\em Bull. Lond. Math. Soc., to appear (arXiv:2305.16788)}, 2023.

\bibitem{ammari2023anderson}
Habib Ammari, Bryn Davies, and Erik~Orvehed Hiltunen.
\newblock Anderson localization in the subwavelength regime.
\newblock {\em Comm. Math. Phys.}, 405(1):Paper No. 1, 20, 2024.

\bibitem{ammari.davies.ea2021}
Habib Ammari, Bryn Davies, and Erik~Orvehed Hiltunen.
\newblock Functional analytic methods for discrete approximations of subwavelength resonator systems.
\newblock {\em Pure Appl. Anal.}, 6(3):873--939, 2024.

\bibitem{cbms}
Habib Ammari, Bryn Davies, and Erik~Orvehed Hiltunen.
\newblock {\em Mathematical theories for metamaterials: From condensed matter theory to subwavelength physics}.
\newblock NSF--CBMS Regional Conf. Ser. American Mathematical Society, Providence, 2025, to appear.

\bibitem{francesco}
Habib Ammari, Francesco Fiorani, and Erik~Orvehed Hiltunen.
\newblock On the validity of the tight-binding method for describing systems of subwavelength resonators.
\newblock {\em SIAM J. Appl. Math.}, 82(4):1611--1634, 2022.

\bibitem{mcmpp}
Habib Ammari, Brian Fitzpatrick, Hyeonbae Kang, Matias Ruiz, Sanghyeon Yu, and Hai Zhang.
\newblock {\em Mathematical and Computational Methods in Photonics and Phononics}, volume 235 of {\em Mathematical Surveys and Monographs}.
\newblock {American Mathematical Society, Providence, RI}, 2018.

\bibitem{ammarihiltunen2020edge}
Habib Ammari and Erik~Orvehed Hiltunen.
\newblock Edge modes in active systems of subwavelength resonators.
\newblock {\em arXiv:2006.05719}, 2020.

\bibitem{erik_JCP}
Habib Ammari and Erik~Orvehed Hiltunen.
\newblock Time-dependent high-contrast subwavelength resonators.
\newblock {\em J. Comput. Phys.}, 445:110594, 2021.

\bibitem{2016NatRM_Alu}
Steven~A. {Cummer}, Johan {Christensen}, and Andrea {Al{\`u}}.
\newblock {Controlling sound with acoustic metamaterials}.
\newblock {\em Nature Reviews Materials}, 1(3):16001, March 2016.

\bibitem{florian}
Florian Feppon and Habib Ammari.
\newblock Subwavelength resonant acoustic scattering in fast time-modulated media.
\newblock {\em J. Math. Pures Appl. (9)}, 187:233--293, 2024.

\bibitem{GaliffiYinAlú+2022+3575+3581}
Emanuele Galiffi, Shixiong Yin, and Andrea Al{\`u}.
\newblock Tapered photonic switching.
\newblock {\em Nanophotonics}, 11(16):3575--3581, 2022.

\bibitem{josa2}
Benjamin~M. Goldsberry, Samuel~P. Wallen, and Michael Haberman.
\newblock Acoustic scattering from spatiotemporally modulated domains.
\newblock {\em J. Acoust. Soc. Amer.}, 154(4-supplement):A57--A57, 2023.

\bibitem{josa1}
Benjamin~M. Goldsberry, Samuel~P. Wallen, and Michael~R. Haberman.
\newblock A green’s function approach for nonreciprocal vibrations of finite elastic structures with spatiotemporally modulated material properties.
\newblock {\em J. Acoust. Soc. Amer.}, 150(4-Supplement):A108--A108, 2021.

\bibitem{erik_bryn_3dscattering}
E.~O. Hiltunen and B.~Davies.
\newblock Coupled harmonics due to time-modulated point scatterers.
\newblock {\em Phys. Rev. B}, 110:184102, Nov 2024.

\bibitem{kadic20193d}
Muamer Kadic, Graeme~W Milton, Martin van Hecke, and Martin Wegener.
\newblock {3D} metamaterials.
\newblock {\em Nat. Rev. Phys.}, 1(3):198--210, 2019.

\bibitem{Ma_Sheng_2016}
Guancong Ma and Ping Sheng.
\newblock Acoustic metamaterials: From local resonances to broad horizons.
\newblock {\em Science Advances}, 2(2):e1501595, 2016.

\bibitem{skin_modulated}
Kei Matsushimaa and Takayuki Yamada.
\newblock Non-bloch band theory for time-modulated discrete mechanical systems.
\newblock {\em arXiv:2407.09871}, 2024.

\bibitem{mdpi_ostyn}
Mark Ostyn, Siyong Kim, and Woon-Hong Yeo.
\newblock A simulation study of a radiofrequency localization system for tracking patient motion in radiotherapy.
\newblock {\em Sensors}, 16(4), 2016.

\bibitem{Ramezani_Jha_Wang_Zhang}
Hamidreza Ramezani, Pankaj~K. Jha, Yuan Wang, and Xiang Zhang.
\newblock Nonreciprocal localization of photons.
\newblock {\em Phys. Rev. Lett.}, 120:043901, Jan 2018.

\bibitem{Sharabi_Lustig_Segev}
Yonatan Sharabi, Eran Lustig, and Mordechai Segev.
\newblock Disordered photonic time crystals.
\newblock {\em Phys. Rev. Lett.}, 126:163902, Apr 2021.

\bibitem{THOMES_spacetimeloc}
Renan~Lima Thomes, Danilo Beli, and Carlos {De Marqui}.
\newblock Space–time wave localization in electromechanical metamaterial beams with programmable defects.
\newblock {\em Mechanical Systems and Signal Processing}, 167:108550, 2022.

\bibitem{fink}
Xinhua Wen, Xinghong Zhu, Alvin Fan, Wing~Yim Tam, Jie Zhu, Hong~Wei Wu, Fabrice Lemoult, Mathias Fink, and Jensen Li.
\newblock Unidirectional amplification with acoustic non-hermitian space-time varying metamaterial.
\newblock {\em Commun. Phys.}, 5, 2022.

\end{thebibliography}
	\bibliographystyle{plain}

\end{document}